\documentclass{amsart}

\usepackage{amsmath,amssymb,amscd,amsthm,latexsym,accents,stmaryrd}
\usepackage{amsfonts}
\usepackage{bbm}
\usepackage[latin1]{inputenc}
\usepackage{graphicx}
\usepackage{float}
\usepackage{subfig}
\usepackage{caption}
\usepackage[all,cmtip]{xy}
\usepackage{enumerate}
\usepackage{amsaddr}
\usepackage{url}
\usepackage{hyperref}
\usepackage{tikz}
\usetikzlibrary{decorations.text}
\usetikzlibrary{patterns}
\usepackage{mathtools}
\usetikzlibrary{automata,positioning}

\newtheorem{theor}{Theorem}[subsection]
\newtheorem{lem}[theor]{Lemma}
\newtheorem{prop}[theor]{Proposition}
\newtheorem{cor}[theor]{Corollary}
\theoremstyle{definition}

\theoremstyle{remark}
\newtheorem{rem}[theor]{Remark}
\newtheorem{rems}[theor]{Remarks}

\newtheorem{ex}[theor]{Example}
\newtheorem{exs}[theor]{Examples}

\newdir{ >}{{}*!/-7pt/\dir{>}}

\def\Z{{\mathbb{Z}}}

\def\im{{\mbox{im}}}

\def\A{{\mathcal{A}}}
\def\B{{\mathcal{B}}}

\def\Hom{{\rm{Hom}}}
\def\deg{{\rm{deg}}}

\def\sgn{{\rm{sgn}}}
\def\char{{\rm{char}}}
\def\sing{{\mathsf{sing}}}
\def\cell{{\mathsf{cell}}}

\allowdisplaybreaks

\begin{document}

\title{Labeled homology of higher-dimensional automata}

\author{Thomas Kahl}

\address{Centro de Matem\'atica,
	Universidade do Minho, Campus de Gualtar,\\
	4710-057 Braga,
	Portugal
}

\thanks{This research was financed by Portuguese funds through FCT - Fundação para a Ciência e a Tecnologia (project UID/MAT/00013/2013).}

\email{kahl@math.uminho.pt}




\begin{abstract}
We construct labeling homomorphisms on the cubical homology of higher-dimensional automata and show that they are natural with respect to cubical dimaps and compatible with the tensor product of HDAs. We also indicate two possible applications of labeled homology in concurrency theory.	
\end{abstract}

\maketitle 

\section{Introduction}

\subsection{Higher-dimensional automata} 
\begin{sloppypar}
Higher-dimensional automata provide a powerful model for concurrent systems. A higher-dimensional automaton (HDA) over a monoid $M$ is a precubical set (i.e., a cubical set without degeneracies) with initial and final states and with 1-cubes labeled by elements of $M$ such that opposite edges of 2-cubes have the same label. The vertices of an HDA represent the states of a concurrent system, the labeled edges model the actions of the system, and a cube of dimension $\geq 2$ indicates that the actions starting at its origin are independent in the sense that they may be executed in any order, or even simultaneously, without any observable difference. The concept of higher-dimensional automaton goes back to Pratt \cite{Pratt}. The notion defined here is essentially the one introduced by van Glabbeek (see \cite{vanGlabbeek}).
\end{sloppypar}

Let us consider the example of Peterson's algorithm \cite{Peterson} to indicate in more detail how HDAs can be used to model concurrent systems. Peterson's algorithm is a protocol designed to give two processes fair and mutually exclusive access to a shared resource. The algorithm is based on three shared variables, namely the boolean variables $\mathtt{b_0}$ and $\mathtt{b_1}$ and the turn variable $\mathtt{t}$, whose possible values are the process IDs---let us assume that these are $0$ and $1$. Process $i$ has four local states and moves from one to the next by executing the following actions: It first sets variable $\mathtt{b_i}$ to $1$ in order to indicate that it intends to enter the ``critical section" and access the shared resource. Then it gives priority to the other process by setting the turn variable $\mathtt{t}$ to $1-i$. After this, it waits until the other process does not intend to enter the critical section or it is its own turn to do so. When this ``guard condition" is satisfied, the process enters the critical section and accesses the shared resource. This action has no effect on the shared variables $\mathtt{b_0}$,  $\mathtt{b_1}$, and $\mathtt{t}$. After having used the shared resource, process $i$ leaves the critical section and sets variable $\mathtt{b_i}$ to $0$ again. The procedure is now repeated arbitrarily often or forever. 

An HDA model of the reachable part of the system given by Peterson's algorithm is depicted in Figure \ref{peterHDA}. The vertices of the HDA correspond to the reachable global states of the system, which are quintuples whose components are  local states of the processes and values of the variables. The HDA has two states, marked with a small incoming arrow and a double circle in Figure \ref{peterHDA}, that are at the same time initial and final states. In the upper initial and final state, $\mathtt{t} = 0$; in the lower one, $\mathtt{t} = 1$. The boolean variables $\mathtt{b_0}$ and $\mathtt{b_1}$ are $0$ in both of these states, and each process is in its initial local state. The directed edges starting in a given vertex represent the actions of the processes that are enabled in the corresponding global state. These actions are indicated in the transition labels, indexed by the respective process IDs. The monoid of labels is the free monoid generated by the transition labels. The squares are introduced in order to indicate independence of actions: two actions are independent in a state where both are enabled if they can be executed one after the other in either order and the two sequential executions of the actions lead to the same state. In general, an HDA may also contain higher-dimensional cubes in order to indicate the independence of more than two actions. A more formal description of the construction of HDA models for shared-variable systems can be found in \cite{transhda}.

\begin{figure}[t]
	\center
	\begin{tikzpicture}[initial text={},on grid]

	\path[draw, fill=lightgray] (0,0)--(1,0)--(1,-1)--(0,-1)--cycle;    
	
	\path[draw, fill=lightgray] (1,0)--(2,0)--(2,-1)--(1,-1)--cycle;
	
	\path[draw, fill=lightgray] (2,0)--(4,0)--(3,-1)--(2,-1)--cycle;
	
	\path[draw, fill=lightgray] (0,-1)--(1,-1)--(1,-3)--(0,-3)--cycle;
	
	\path[draw, fill=lightgray] (1,-1)--(2,-1)--(2,-2)--(1,-3)--cycle;
	
	\path[draw, fill=lightgray] (3,-1)--(4,0)--(4,-2)--(3,-2)--cycle;
	
	\path[draw, fill=lightgray] (4,0)--(5,0)--(5,-2)--(4,-2)--cycle;
	
	\path[draw, fill=lightgray] (2,-2)--(3,-2)--(3,-3)--(1,-3)--cycle;
	
	\path[draw, fill=lightgray] (3,-2)--(4,-2)--(4,-3)--(3,-3)--cycle;
	
	\path[draw, fill=lightgray] (4,-2)--(5,-2)--(5,-3)--(4,-3)--cycle;
	
	\draw[red,very thick] (3,-1) to [out=225,in=45] (-1.35,-1.65); 
	
	\draw[->,red,very thick] (-1.35,-1.65) to [out=225,in=180] (-0.1,-4);
	
	\node at (-2,-3) {\scalebox{0.85}{$\mathtt{t\!\coloneqq_0\! 1}$}};
	
	\draw[red,very thick] (2,-2) to [out=45,in=225] (6.35,-1.35); 
	
	\draw[->,red,very thick] (6.35,-1.35) to [out=45,in=0] (5.1,1);

	\node at (7,0) {\scalebox{0.85}{$\mathtt{t\!\coloneqq_1\! 0}$}};

	\node[state,minimum size=0pt,inner sep =2pt,fill=white] (q_0) at (0,0)  {}; 
	
	\node[state,minimum size=0pt,inner sep =2pt,fill=white,accepting, initial,initial where=above,initial distance=0.2cm] (q_1) [right=of q_0,xshift=0cm] {};
	
	\node[state,minimum size=0pt,inner sep =2pt,fill=white] (q_2) [right=of q_1,xshift=0cm] {};
	
	\node[state,minimum size=0pt,inner sep =2pt,fill=white] (q_3) [right=of q_2,xshift=1cm] {};
	
	\node[state,minimum size=0pt,inner sep =2pt,fill=white] (q_4) [right=of q_3,xshift=0cm] {};
	
	\node[state,minimum size=0pt,inner sep =2pt,fill=white] (q_5) [below=of q_0,xshift=0cm] {};  
	
	\node[state,minimum size=0pt,inner sep =2pt,fill=white] (q_6) [right=of q_5,xshift=0cm] {};
	
	\node[state,minimum size=0pt,inner sep =2pt,fill=white] (q_7) [right=of q_6,xshift=0cm] {};
	
	\node[state,minimum size=0pt,inner sep =2pt,fill=white] (q_8) [right=of q_7,xshift=0cm] {};
	
	\node[state,minimum size=0pt,inner sep =2pt,fill=white] (q_9) [below=of q_7,xshift=0cm] {};
	
	\node[state,minimum size=0pt,inner sep =2pt,fill=white] (q_10) [right=of q_9,xshift=0cm] {};
	
	\node[state,minimum size=0pt,inner sep =2pt,fill=white] (q_11) [right=of q_10,xshift=0cm] {};
	
	\node[state,minimum size=0pt,inner sep =2pt,fill=white] (q_12) [right=of q_11,xshift=0cm] {};
	
	\node[state,minimum size=0pt,inner sep =2pt,fill=white] (q_17) [below=of q_12,xshift=0cm] {};
	
	\node[state,minimum size=0pt,inner sep =2pt,fill=white,accepting,initial,initial where=below,initial distance=0.2cm] (q_16) [left=of q_17,xshift=0cm] {};
	
	\node[state,minimum size=0pt,inner sep =2pt,fill=white] (q_15) [left=of q_16,xshift=0cm] {};
	
	\node[state,minimum size=0pt,inner sep =2pt,fill=white] (q_14) [left=of q_15,xshift=-1cm] {};
	
	\node[state,minimum size=0pt,inner sep =2pt,fill=white] (q_13) [left=of q_14,xshift=0cm] {};
	
	\node[state,minimum size=0pt,inner sep =2pt,fill=white] (p_0) [above=of q_4,xshift=0cm] {};
	
	\node[state,minimum size=0pt,inner sep =2pt,fill=white] (p_1) [below=of q_13,xshift=0cm] {};

	\path[->] 
	(p_0) edge[right,red,very thick] node[black] {\scalebox{0.85}{$\mathtt{crit_0}$}} (q_4)
	
	(q_0) edge[below,green] node[black] {\scalebox{0.85}{$\mathtt{b_1\!\!\coloneqq_1\!\! 0}$}} (q_1)
	(q_1) edge[below,green] node[black] {\scalebox{0.85}{$\mathtt{b_1\!\!\coloneqq_1\!\! 1}$}} (q_2)
	(q_2) edge[below,green] node[black] {\scalebox{0.85}{$\mathtt{t\!\!\coloneqq_1\!\! 0}$}} (q_3)
	(q_4) edge[above,red] node[black] {\scalebox{0.85}{$\mathtt{b_0\!\!\coloneqq_0\!\! 0}$}} (q_3)
	
	(q_0) edge[left,very thick,blue] node[black] {\scalebox{0.85}{$\mathtt{b_0\!\!\coloneqq_0\!\! 1}$}} (q_5)
	(q_1) edge[above] node {} (q_6)
	(q_2) edge[above] node {} (q_7)
	(q_3) edge[above,red,very thick] node {} (q_8)
	
	(q_5) edge[above] node {} (q_6)
	(q_6) edge[above] node {} (q_7)
	(q_7) edge[above] node {} (q_8)
	
	(q_7) edge[above] node {} (q_9)
	(q_10) edge[above] node {} (q_8)
	(q_11) edge[above] node {} (q_3)
	(q_12) edge[right, very thick,blue] node[black] {\scalebox{0.85}{$\mathtt{t\!\!\coloneqq_1\!\! 0}$}} (q_4)
	
	(q_10) edge[above] node {} (q_9)
	(q_11) edge[above] node {} (q_10)
	(q_12) edge[above] node {} (q_11)
	
	(q_5) edge[left,very thick, blue] node[black] {\scalebox{0.85}{$\mathtt{t\!\!\coloneqq_0\!\! 1}$}} (q_13)
	(q_6) edge[above] node {} (q_14)
	(q_14) edge[above,red,very thick] node {} (q_9)
	(q_15) edge[above] node {} (q_10)
	(q_16) edge[above] node {} (q_11)
	(q_17) edge[right,very thick,blue] node[black] {\scalebox{0.85}{$\mathtt{b_1\!\!\coloneqq_1\!\! 1}$}} (q_12)
	
	(q_13) edge[below,red] node[black] {\scalebox{0.85}{$\mathtt{b_1\!\!\coloneqq_1\!\! 0}$}} (q_14)
	(q_15) edge[above,green] node[black] {\scalebox{0.85}{$\mathtt{t\!\!\coloneqq_0\!\! 1}$}} (q_14)
	(q_16) edge[above,green] node[black] {\scalebox{0.85}{$\mathtt{b_0\!\!\coloneqq_0\!\! 1}$}} (q_15)
	(q_17) edge[above,green] node[black] {\scalebox{0.85}{$\mathtt{b_0\!\!\coloneqq_0\!\! 0}$}} (q_16)
	
	(p_1) edge[left,red,very thick] node[black] {\scalebox{0.85}{$\mathtt{crit_1}$}} (q_13)
	
	(q_3) edge[above, bend right,green,very thick] node[black] {\scalebox{0.85}{$\mathtt{crit_1}$}} (q_0)
	(q_14) edge[below, bend right,green,very thick] node[black] {\scalebox{0.85}{$\mathtt{crit_0}$}} (q_17)
	;

	\end{tikzpicture}
	\caption{HDA for Peterson's algorithm with colored homology generators. Parallel arrows are supposed to have the same label.}\label{peterHDA}
\end{figure}
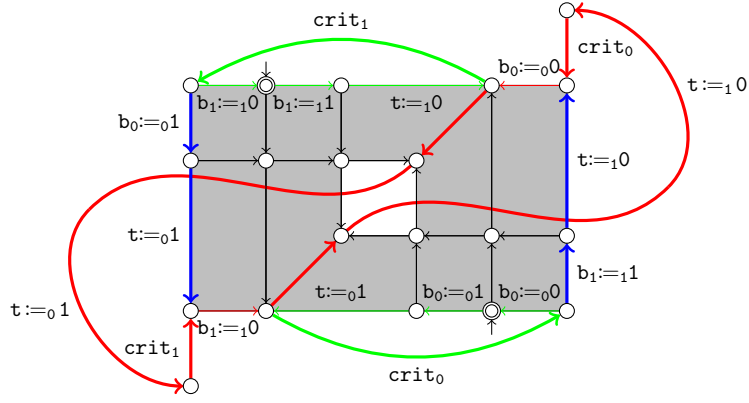

\subsection{Labeled homology} As demonstrated in the literature, concepts and methods from algebraic topology may be employed profitably in concurrency theory (see e.g. \cite{FajstrupGR, FGHMR}). The purpose of this paper is to introduce \emph{labeled homology} of HDAs and to establish results about it that enable one to use the information contained in the homology of HDAs in the analysis of concurrent systems. 

We focus on HDAs over the free monoid on an alphabet $\Sigma$. Given such an HDA $\A$, we consider its cubical homology and the exterior algebra $\Lambda(\Sigma)$ and show that the labeling function of $\A$ induces \emph{labeling homomorphisms} 
\[\ell_{\A} = \ell_{\A}^n \colon H_n(\A) \to \Lambda(\Sigma) \quad (n\geq 0).\]
The labeled homology of an HDA is its homology together with the sequence of labeling homomorphisms.

Let us describe the labeled homology of the HDA $\A$ modeling the system given by Peterson's algorithm. The integral homology of $\A$ is given by $H_0(\A) \cong \Z$, $H_1(\A) \cong \Z^5$, and $H_{\geq 2}(\A) = 0$. The labeling homomorphism in degree 0 is the obvious monomorphism   $H_0(\A) \to \Lambda(\Sigma)$, and so the only actually interesting dimension is 1. If, as in $\A$, the edge labels of an HDA are indecomposable, the label of a $1\mbox{-}$di\-men\-sional homology class is simply the linear combination of labels corresponding to the linear combination of edges in a representing cycle of the class. Hence the labels of the two generators $\alpha_0$ and $\alpha_1$ of $H_1(\A)$  depicted in green in Figure \ref{peterHDA} are given by 
\[
\ell_{\A}(\alpha_i) = \mathtt{b_i\!\!\coloneqq_i\!\!1} + \mathtt{t\!\!\coloneqq_i\!\!(1-i)} + \mathtt{crit_i} + \mathtt{b_i\!\!\coloneqq_i\!\!0}.
\] 
The classes $\alpha_0$ and $\alpha_1$ represent the two processes, executing alone. A third generator of $H_1(\A)$, depicted in red in Figure \ref{peterHDA}, has label $\ell_{\A}(\alpha_0)+\ell_{\A}(\alpha_1)$ and corresponds to an execution where the two processes alternately access the shared resource. Finally, $H_1(\A)$ has two generators with zero label, which are represented by the thick multicolored cycles in Figure \ref{peterHDA} and reflect the synchronization (coordination) of the processes enforced by the guard condition that must be satisfied before a process enters the critical section: The existence of, say, the left of these homology classes is due the fact that process 1 is blocked in the upper right state of the inner hole of $\A$. If process 1 was allowed to enter the critical section in this state, the class would disappear---and the system would lose the important property of starvation freedom as it would become possible that process 0 requests access to the critical section without ever obtaining it.

The example of Peterson's algorithm shows that the homology of an HDA has meaning from a computer science point of view. The labeling homomorphisms help to interpret the information contained in the homology of HDAs. We introduce labeled homology in Section \ref{secLabels}. The earlier sections contain preparatory material on precubical sets and HDAs (Section \ref{secHDA}) and cubical chains and homology  (Section \ref{secCubchains}).

\subsection{Tensor-product HDAs} An important property of labeled homology is that it is compatible with the tensor product of HDAs, which models the interleaving parallel composition of independent concurrent systems. If two completely independent and disjoint concurrent systems are combined into one, then an HDA model of this interleaving parallel composition of the systems is given by the tensor product of the HDA models of the two systems (see \cite{transhda}). In Section \ref{secTensor} we establish a cross product formula for the computation of the labeling homomorphisms of a tensor-product HDA from those of its factors.

A concurrent system will, of course, normally not be the interleaving  of other systems. It is, however, frequently the case that a concurrent system contains subsystems that are independent from each other. Labeled homology and, in particular, our result on the labeled homology of tensor-product HDAs can be used to aid the analysis of the independence structure of concurrent systems. This will be discussed in Section \ref{secIndep}.

\subsection{Cubical dimaps}

\begin{figure}[t]
	\begin{tikzpicture}[initial text={},on grid]

	\draw[] (4,0) to [out=225,in=45] (-0.35,-1.65); 
	
	\draw[->] (-0.35,-1.65) to [out=225,in=180] (0.9,-3);
	
	\node[align=left] at (-1,-2) {\scalebox{0.85}{$\mathtt{b_0\!\!\coloneqq_0\!\! 1;}$}\\\scalebox{0.85}{$\mathtt{t\!\!\coloneqq_0\!\! 1;}$}\\\scalebox{0.85}{$\mathtt{crit_1;}$}\\\scalebox{0.85}{$\mathtt{b_1\!\!\coloneqq_1\!\! 0}$}};

	\draw[] (1,-3) to [out=45,in=225] (5.35,-1.35); 
	
	\draw[->] (5.35,-1.35) to [out=45,in=0] (4.1,0);
	
	\node[align=left] at (6.1,-1) {\scalebox{0.85}{$\mathtt{b_1\!\!\coloneqq_1\!\! 1;}$}\\\scalebox{0.85}{$\mathtt{t\!\!\coloneqq_1\!\! 0;}$}\\\scalebox{0.85}{$\mathtt{crit_0;}$}\\\scalebox{0.85}{$\mathtt{b_0\!\!\coloneqq_0\!\! 0}$}};
	
	\node[align=left] at (1.6,-1.8) {\scalebox{0.85}{$\mathtt{b_0\!\!\coloneqq_0\!\! 1;}$}\\\scalebox{0.85}{$\mathtt{t\!\!\coloneqq_0\!\! 1}$}}; 
	
	\node[align=left] at (3.4,-1.3) {\scalebox{0.85}{$\mathtt{b_1\!\!\coloneqq_1\!\! 1;}$}\\\scalebox{0.85}{$\mathtt{t\!\!\coloneqq_1\!\! 0}$}};

	\node[state,minimum size=0pt,inner sep =2pt,fill=white,accepting,initial,initial where=above,initial distance=0.2cm] (q_1) at (1,0) {};

	\node[state,minimum size=0pt,inner sep =2pt,fill=white] (q_3) [right=of q_2,xshift=1cm] {};

	\node[state,minimum size=0pt,inner sep =2pt,fill=white,accepting,initial,initial where=below,initial distance=0.2cm] (q_16) [left=of q_17,xshift=0cm] {};

	\node[state,minimum size=0pt,inner sep =2pt,fill=white] (q_14) [left=of q_15,xshift=-1cm] {};

	\path[->] 
	
	(q_1) edge[below] node {\scalebox{0.85}{$\mathtt{b_1\!\!\coloneqq_1\!\! 1;t\!\!\coloneqq_1\!\! 0}$}} (q_3)

	(q_1) edge[right] node{}(q_14)

	(q_16) edge[right] node {} (q_3)

	(q_16) edge[above] node {\scalebox{0.85}{$\mathtt{b_0\!\!\coloneqq_0\!\! 1;t\!\!\coloneqq_0\!\! 1}$}} (q_14)

	(q_3) edge[above, bend right] node {\scalebox{0.85}{$\mathtt{crit_1;b_1\!\!\coloneqq_1\!\! 0}$}} (q_1)
	(q_14) edge[below, bend right] node {\scalebox{0.85}{$\mathtt{crit_0;b_0\!\!\coloneqq_0\!\! 0}$}} (q_16)
	;
	\end{tikzpicture}
	\caption{Small HDA model for Peterson's algorithm}\label{peterHDAmin}
\end{figure}
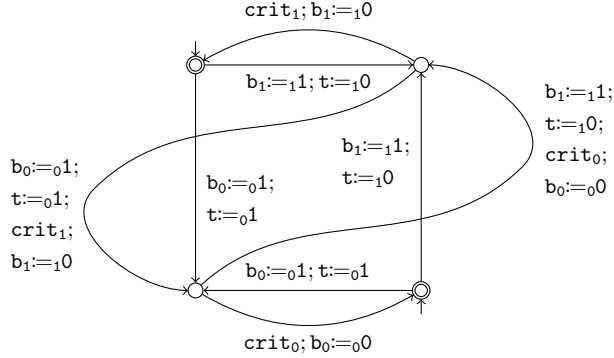

In Section \ref{secNat} we show that labeled homology is natural with respect not only to morphisms of HDAs but also to \emph{cubical dimaps}. Roughly speaking, a cubical dimap from an HDA to another one is a label-preserving continuous map between the geometric realizations that also is a directed map---or dimap---in the sense that it respects the directed topology of the geometric realization.
Cubical dimaps permit one to compare HDAs that model concurrent systems at different abstraction levels and do not admit ordinary  morphisms between them. For instance, although the HDA in Figure \ref{peterHDA} and the homotopy equivalent smaller one depicted in Figure \ref{peterHDAmin} both model the system given by Peterson's algorithm (cf. \cite[7.2]{topabs}), there do not exist any morphisms from one to the other. The two HDAs may, however, be related by a cubical dimap. Cubical dimaps are the subject of Section \ref{secDimaps}.

The naturality of labeled homology implies that it is invariant under cubical dimaps that are homotopy equivalences. Therefore labeled homology can be considered a directed homotopy invariant of HDAs. Since it does not only depend on the directed topology of the underlying precubical sets, it is, however, not a concept of directed homology as those considered, for instance, in \cite{DubutGG,FahrenbergDiH,GaucherHomol,GoubaultJensen,GrandisBook,hgraph,KrishnanFlowCut}.

As a possible application of labeled homology, based on its naturality, we discuss in Section \ref{secSpec} a necessary condition for an HDA to implement a specification that is given by another HDA.

\section{Precubical sets and HDAs} \label{secHDA}

This section presents some fundamental concepts and facts about precubical sets and higher-dimensional automata. The material is mostly taken from \cite{weakmor}.

\subsection{Precubical sets} \label{precubs}

A \emph{precubical set} is a graded set $P = (P_n)_{n \geq 0}$ with  \emph{boundary} or \emph{face operators} $d^k_i\colon P_n \to P_{n-1}$ $(n>0,\;k= 0,1,\; i = 1, \dots, n)$ satisfying the relations $d^k_i\circ d^l_{j}= d^l_{j-1}\circ d^k_i$ $(k,l = 0,1,\; i<j)$. If $x\in P_n$, we say that $x$ is of  \emph{degree} or \emph{dimension} $n$ and write $\deg(x) = n$. The elements of degree $n$ are called the \emph{$n$-cubes} of $P$. The elements of degree $0$ are also called the \emph{vertices} of $P$, and the $1$-cubes are also called the \emph{edges} of $P$. Given an $n$-cube $x$, we say that the vertex $\underbracket[0.5pt]{d_1^{0} \cdots d_1^{0}}_{n\;   \text{times}}x$ is its \emph{initial vertex} and that $d_1^{1} \cdots d_1^{1}x$ is its \emph{final vertex}. We say that a face $d^k_ix$ of a cube $x$ is  a \emph{front (back) face} of $x$ if $k=0$ $(k=1)$. A \emph{morphism} of precubical sets is a morphism of graded sets that is compatible with the boundary operators. A \emph{precubical subset} of a precubical set $P$ is a graded subset of $P$ that is stable under the boundary operators.

\subsection{Intervals} 
\begin{sloppypar}
Let $k$ and $l$ be two integers such that $k \leq l$. The \emph{precubical interval}  $\llbracket k,l \rrbracket$ is the  precubical set defined by $\llbracket k,l \rrbracket_0 = \{k,\dots , l\}$, $\llbracket k,l \rrbracket_1 =  \{{[k,k+ 1]}, \dots , {[l- 1,l]}\}$, $d_1^0[j-1,j] = j-1$, $d_1^1[j-1,j] = j$, and $\llbracket k,l \rrbracket_{\geq 2} = \emptyset$. 
\end{sloppypar} 

\subsection{Tensor product}
The category of precubical sets is a monoidal category. Given two precubical sets $P$ and $Q$, the tensor product $P\otimes Q$ is  defined by \[(P\otimes Q)_n = \coprod \limits_{p+q = n} P_p\times Q_q\] and
$$d_i^k(x,y) = \left\{ \begin{array}{ll} (d_i^kx,y), & 1\leq i \leq \deg(x),\\
(x,d_{i-\deg(x)}^ky), & \deg(x)+1 \leq i \leq \deg(x) + \deg(y).
\end{array}\right.$$ 
The $n$-fold tensor product of a precubical set $P$ is denoted by $P^{\otimes n}$. Here we use the convention $P^{\otimes 0} = \llbracket 0,0\rrbracket = \{0\}$. 

\subsection{Precubical cubes}
\begin{sloppypar} 
The \emph{precubical $n\mbox{-}$cube} is the $n$-fold tensor product ${\llbracket 0,1\rrbracket^ {\otimes n}}$. The only element of degree $n$ in $\llbracket 0,1\rrbracket^ {\otimes n}$ will be denoted by $\iota_n$. We thus have $\iota_0 = 0$ and $\iota_n = (\underbracket[0.5pt]{ [0,1] ,\dots , [ 0,1]}_{n\; {\text{times}}})$ for $n>0$. Given an element $x$ of degree $n$ of a precubical set $P$, there exists a unique morphism of precubical sets $\llbracket 0,1\rrbracket ^{\otimes n}\to P$ that sends $\iota_n$ to $x$. This morphism will be denoted by $x_{\sharp}$.
\end{sloppypar}

\subsection{Paths} 
A \emph{path of length $k$} in a precubical set $P$ is a morphism of precubical sets $\omega \colon \llbracket 0,k \rrbracket \to P$. If $\omega(0) = \omega(k)$, we say that $\omega$ is a \emph{loop}. The set of paths in $P$ is denoted by $P^{\mathbb I}$. The \emph{concatenation} $\omega \cdot\nu$  of two paths $\omega \colon \llbracket 0,k \rrbracket \to P$ and $\nu \colon \llbracket 0,l \rrbracket \to P$ with $\omega (k) = \nu (0)$ is defined in the obvious way. Note that for any path $\omega \in P^{\mathbb I}$ of length $k \geq 1$, there exists a unique sequence $(x_1, \dots , x_k)$ of elements of $P_1$ such that $d_1^0x_{j+1} = d_1^1x_j$ for all $1\leq j < k$ and $\omega = x_{1\sharp} \cdot\, \cdots \,\cdot x_{k\sharp}$.

\subsection{Geometric realization} 
The \emph{geometric realization} of a precubical set $P$ is the quotient space \[|P|=\left(\coprod _{n \geq 0} P_n \times [0,1]^n\right)/\sim\] where the sets $P_n$ are given the discrete topology and the equivalence relation is given by
\[(d^k_ix,u) \sim (x,\delta_i^k(u)), \quad  x \in P_{n+1},\; u\in [0,1]^n,\; i \in  \{1, \dots, n+1\},\; k \in \{0,1\}.\] 
Here the continuous maps $\delta_i^ k \colon [0,1]^ n \to [0,1]^ {n+1}$ are defined by \[\delta_i^ k(u_1, \dots, u_n) = (u_1, \dots, u_{i-1},k,u_i, \dots, u_n).\]
The geometric realization of a morphism of precubical sets $f\colon P \to Q$ is the continuous map $|f|\colon |P| \to |Q|$ given by $|f|([x,u])= [f(x),u]$.
 
The geometric realization of a precubical set $P$ is a CW complex. The $n$-skeleton of $|P|$ is the geometric realization of the precubical subset $P_{\leq n}$ of $P$ defined by $(P_{\leq n})_m = P_m$ $(m\leq n)$ and $(P_{\leq n})_m = \emptyset$ $(m> n)$. As a CW complex, the geometric realization of a precubical set is a compactly generated Hausdorff space. All spaces in this paper are compactly generated Hausdorff spaces, and constructions such as products are performed in the category of these spaces.

The geometric realization is a comonoidal functor with respect to the natural continuous map ${\psi_{P,Q}\colon |P\otimes Q| \to |P| \times |Q|}$ given by 
\begin{align*}
\MoveEqLeft{\psi_{P,Q} ([(x,y),(u_1,\dots ,u_{\deg(x)+\deg(y)})])}\\ &= ([x,(u_1,\dots , u_{\deg(x)})],[y,(u_{\deg(x)+1}, \dots u_{\deg(x)+\deg(y)}]).
\end{align*} 
Since we are working in the category of compactly generated Hausdorff spaces, the map $\psi_{P,Q}$ is a homeomorphism, even a cellular isomorphism, and we may use it to identify the spaces $|P\otimes Q|$ and $|P| \times |Q|$. 

The geometric realization of the precubical interval $\llbracket k,l \rrbracket$ will be identified with the closed interval $[k,l]$ by means of the homeomorphism $|\llbracket k,l \rrbracket| \to [k,l]$ given by $[j,()] \mapsto j$ and $[[j-1,j],t] \mapsto j-1+t$.

\subsection{Singular precubical sets} The \emph{singular precubical set} of a space $X$ is the precubical set $SX$ where the $n$-cubes are the continuous maps $\sigma \colon [0,1]^n\to X$ and the boundary operators are given by $d_i^k\sigma = \sigma \circ \delta_i^k$. The singular precubical set is functorial in the obvious way, and the functor $S$ is right adjoint to the geometric realization. 

\subsection{Higher-dimensional automata} \label{HDAdef}

Let $M$ be a monoid. A \emph{higher-dimensional automaton over} $M$ (abbreviated $M\mbox{-}$HDA or simply HDA) is a tuple \[\A = (P_{\A},I_{\A},F_{\A}, \lambda_{\A})\] where $P_{\A}$ is a precubical set, $I_{\A} \subseteq (P_{\A})_0$ is a set of \emph{initial states}, $F_{\A} \subseteq (P_{\A})_0$ is a set of \emph{final} or \emph{accepting  states}, and $\lambda \colon (P_{\A})_1 \to M$ is a map, called the \emph{labeling function}, such that $\lambda_{\A} (d_i^0x) = \lambda_{\A} (d_i^1x)$ for all $x \in (P_{\A})_2$ and $i \in \{1,2\}$ (cf. \cite{vanGlabbeek, weakmor}). We shall abuse notation and write $\A$ instead of $P_{\A}$. A \emph{morphism} from an $M\mbox{-}$HDA $\B$ to an $M\mbox{-}$HDA $\A$ is a morphism of precubical sets  $f\colon \B \to \A$  such that $f(I_{\B}) \subseteq I_{\A}$, $f(F_{\B}) \subseteq F_{\A}$, and  $\lambda_{\A}(f(x)) = \lambda_{\B}(x)$ for all $x \in \B_1$. An $M$-HDA $\B$ is a \emph{subautomaton} of an $M$-HDA $\A$ if $\B$ is a precubical subset of $\A$, $I_{\B} \subseteq I_{\A}$, $F_{\B}\subseteq F_{\A}$, and $\lambda_{\B}(x) = \lambda_{\A}(x)$ for all $x\in {\B_1}$.

\subsection{Extended labeling function} 

The \emph{extended labeling function} of an $M\mbox{-}$HDA $\A$ is the map $\overline{\lambda}_{\A} \colon \A ^{\mathbb I} \to M$ defined as follows: If $\omega = x_{1\sharp} \cdot\, \cdots \,\cdot x_{k\sharp}$ for a sequence $(x_1, \dots , x_k)$ of elements of $\A_1$ such that $d_1^0x_{j+1} = d_1^1x_j$ $(1\leq j < k)$, then we set \[\overline{\lambda}_{\A} (\omega) = \lambda_{\A} (x_1) \cdot\, \cdots \,\cdot \lambda_{\A} (x_k);\] if $\omega$ is a path of length $0$, then we set  $\overline{\lambda}_{\A} (\omega ) = 1$.

\section{The cubical chain complex of a precubical set} \label{secCubchains}

Throughout this article we work over a fixed principal ideal domain  $R$. In order to simplify the presentation, we will usually suppress the ring $R$ from the notation. All graded modules are $\Z$-graded, and we adopt  the usual convention $V^i = V_{-i}$. 

In this section we define the cubical chain complex and the cubical homology and cohomology of a precubical set and relate cubical chains to singular and cellular chains.

\subsection{Cubical (co)chains and (co)homology} Let $P$ be a precubical set. The \emph{cubical chain complex} of $P$ is the nonnegative chain complex $C_*(P)$ where $C_n(P)$ is the free module generated by $P_n$ and the boundary operator $d\colon C_n(P) \to C_{n-1}(P)$  is given by \[dx = \sum \limits_{i=1}^{n}(-1)^i(d^0_ix -d^1_ix), \quad x \in P_n.\]  
It is clear that the cubical chain complex is functorial. The \emph{cubical homology} of $P$, denoted  by $H_*(P)$, is the homology of $C_*(P)$. 

Given a module $G$, we view it as a chain complex concentrated in degree $0$ and define the \emph{cubical cochain complex} of $P$ \emph{with coefficients in $G$}, $C^*(P;G)$, to be the Hom complex $\Hom (C_*(P),G)$. The \emph{cubical cohomology of P with coefficients in $G$} is the graded module  $H^*(P;G) = H(C^*(P;G))$.

We will not need to consider cubical chains and homology with coefficients in a module.

\begin{rem}
An interesting elementary fact about $H_*(P)$ is that for every loop $\omega = x_{1\sharp} \cdots x_{k\sharp} \in P^{\mathbb I}$ such that $\char(R)\nmid k$, the homology class of the cycle $z_{\omega} = \sum_{i=1}^k x_i$ is nonzero. Indeed, if $\nu \in C^1(P;R)$ is the cochain where $\nu(x) = 1$ for all $x \in P_1$, then $\nu(z_{\omega}) = \sum_{i=1}^k \nu(x_i) = k\cdot1 \not= 0$. By definition of the cubical boundary operator, $\nu$ is a cocycle, and therefore $z_{\omega}$ is not a boundary.

\end{rem}

\subsection{Cubical singular chains} Let us briefly relate the cubical chain complex of a precubical set to the singular chain complex of its geometric realization. The \emph{normalized cubical singular chain complex} of a space $X$ is the quotient chain complex $S_*(X) = C_*(SX)/D_*(SX)$ where $D_*(SX)$ denotes the subcomplex of $C_*(SX)$ generated by the degenerate singular cubes \cite{HiltonWylie, Massey}. This chain complex is naturally chain homotopy equivalent to the the usual simplicial singular chain complex of $X$ \cite[Theorem 8.4.7]{HiltonWylie}. The following theorem can, for instance, be established by adapting the proof of \cite[Theorem 2.27]{Hatcher}, which is an analogous result for $\Delta$-complexes.

\begin{theor} \label{natquism}
Let $P$ be a precubical set. A natural chain homotopy equivalence ${\sing\colon C_*(P)} \to S_*(|P|)$ is given by \[\sing(x) = |x_{\sharp}| + D_*(S|P|), \quad x \in P.\] 
\end{theor}

\subsection{Cellular chains}
\begin{sloppypar}
Recall that the \emph{cellular chain complex} of a CW complex $X$ with skeleta $X_ n$, $C^{CW}_*(X)$, is defined using (cubical or simplicial) singular homology as follows (see e.g. \cite[V.7]{Massey}). For $n \geq 0$, one sets
\[
C^{CW}_n(X) = H_n(X_ n,X_{n-1}),
\]
where, by convention, $X_{-1} = \emptyset$. The boundary operator $d\colon C^{CW}_n(X) \to C^{CW}_{n-1}(X)$ is the composite
\[
H_n(X_n,X_{n-1}) \xrightarrow{\partial_*} H_{n-1}(X_{n-1}) \xrightarrow{} H_{n-1}(X_{n-1},X_{n-2})
\]
where $\partial_*$ is the connecting homomorphism in the long exact homology sequence of the pair $(X_n,X_{n-1})$ and the second map is the one occurring in the long exact sequence of the pair $(X_{n-1},X_{n-2})$.

\begin{theor} \label{cubCW}
A natural chain isomorphism  
$\cell \colon C_*(P) \to C^{CW}_*(|P|)$ is given by \[\cell (x) = [|x_{\sharp}|], \quad x \in P_n\]
or, more precisely, by $\cell (x)  = [(|x_{\sharp}|+D_n(S|P_{\leq n}|))+S_n(|P_{<n}|)]$.
\end{theor}

\begin{proof}
For each $n \geq 0$, the quotient chain complex $C_*(P_{\leq n})/C_*(P_{< n})$ is concentrated in degree $n$, and we have a natural isomorphism
\[
C_n(P) \to  (C_*(P_{\leq n})/C_*(P_{< n}))_n \to H_n(P_{\leq n}, P_{< n}).
\]
The map $\cell \colon C_n(P) \to C^{CW}_n(|P|)$ is the composite of this isomorphism with the natural map 
\[\sing_* \colon  H_n(P_{\leq n}, P_{< n}) \to H_n(|P_{\leq n}|,|P_{<n}|),\]
which is an isomorphism by Theorem \ref{natquism}. Consider the following commutative diagram of modules:	
\[
\xymatrix{
C_n(P) \ar^{d}[rr] \ar_{\cong} [d] && C_{n-1}(P) \ar^{\cong} [d]\\
H_n(P_{\leq n},P_{<n}) \ar^(0.55){\partial_*}[r] \ar^{\sing_*}_{\cong} [d] & H_{n-1}(P_{<n}) \ar[r] \ar^{\cong}_{\sing_*} [d] & H_{n-1}(P_{<n},P_{<n-1}) \ar^{\cong}_{\sing_*} [d] \\
H_n(|P_{\leq n}|,|P_{<n}|) \ar_(0.55){\partial_*}[r]  & H_{n-1}(|P_{<n}|) \ar[r]  & H_{n-1}(|P_{< n}|,|P_{<n-1}|)
}
\]	
Since the bottom row is the boundary operator $d\colon C^{CW}_n(|P|) \to C^{CW}_{n-1}(|P|)$, the map $\cell \colon C_*(P) \to C^{CW}_*(|P|)$ is a natural isomorphism of chain complexes.	
\end{proof}
\end{sloppypar}

\section{Labeled homology of an HDA} \label{secLabels}

We use the labeling function of an HDA over a free monoid to define its labeling cochains.  The compatibility between the labeling function and the boundary operators---the fact that opposite edges of a $2$-cube have the same label---permits us to show that these cochains are cocycles and induce a labeling in cubical homology.

\subsection{Labels of parallel edges} 

The following lemma establishes that parallel edges in subdivided cubes of HDAs (i.e., images of tensor products of precubical intervals under morphisms of precubical sets) have the same label:

\begin{lem} \label{paralleledges} Let $\A$ be an $M\mbox{-}$HDA, let $\chi \colon {\llbracket 0, l_1\rrbracket \otimes \cdots \otimes \llbracket 0, l_n\rrbracket} \to \A$ be a morphism of precubical sets, and let $1\leq s \leq n$ and $0 \leq j_s < l_s$    be integers. Then for all tuples $(j_1,\dots, j_{s-1},j_{s+1},\dots ,j_n)$ with $j_i \in \{0, \dots, l_i\}$,
	\[
	\lambda_{\A} (\chi (j_1, \dots, j_{s-1},[j_s,j_s+1],j_{s+1}, \dots, j_n)) = \lambda_{\A} (\chi (0, \dots, 0,[j_s,j_s+1],0, \dots, 0)).
	\] 
\end{lem}

\begin{proof}
	Order the tuples $(j_1,\dots, j_{s-1},j_{s+1},\dots ,j_n)$ lexicographically, and suppose inductively that the statement holds for all predecessors of a given such tuple that has at least one nonzero component. Let $t\not= s$ be an index such that $j_t\not=0$, and consider the tuple $(j'_1,\dots, j'_{s-1},j'_{s+1},\dots ,j'_n)$ where 
	\[ j'_i = \left \{ \begin{array}{ll} j_i, & i \not=t,\\
	j_t-1, &  i=t. \end{array} \right. \]
	This is a predecessor of the given tuple $(j_1,\dots, j_{s-1},j_{s+1},\dots ,j_n)$. Consider the $2$-cube $b \in \llbracket 0, l_1\rrbracket \otimes \cdots \otimes \llbracket 0, l_n\rrbracket$ defined by 
	\[ b_i = \left \{ \begin{array}{ll} \left[j_s,j_s+1\right], & i = s,\\
	\left[j_t-1,j_t\right], & i = t,\\
	j_i, & i \not= s,t. \end{array} \right. \]
	Set 
	\[ q = \left \{ \begin{array}{ll} 1, & s > t,\\
	2, &  s< t. \end{array} \right. \]
	Then 
	\begin{align*}
		\MoveEqLeft{
			\lambda_{\A} (\chi (j_1, \dots, j_{s-1},[j_s,j_s+1],j_{s+1}, \dots, j_n))}\\
		&= \lambda_{\A} (\chi (d^1_qb))\\
		&= \lambda_{\A} (d^1_q\chi (b))\\
		&= \lambda_{\A} (d^0_q\chi (b))\\
		&= \lambda_{\A} (\chi (d^0_qb))\\
		&= \lambda_{\A} (\chi (j'_1, \dots, j'_{s-1},[j_s,j_s+1],j'_{s+1}, \dots, j'_n))\\
		&= \lambda_{\A} (\chi (0, \dots, 0,[j_s,j_s+1],0, \dots, 0)). \qedhere
	\end{align*}	
\end{proof}

\subsection{The edge  \texorpdfstring{$e^k_ix$}{}} \label{eki}

Let $x$ be an element of degree $n > 0$ of a precubical set $P$, and let $k\in\{0,1\}$ and $i \in \{1, \dots, n\}$. We define the element $e^k_ix \in P_1$ by 
\[e^ k_ix = \left\{\begin{array}{ll} x,& n = 1,\\d_1^{1-k} \cdots d_{i-1}^ {1-k}d_{i+1}^{1-k}\cdots d_n^{1-k}x, & n > 1. \end{array} \right.\] The notation, introduced in \cite{topabs}, reflects the fact that the edge $e^k_ix$ is associated with the face $d^k_ix$: The element $e^0_ix$ is an edge of $x$ leading from the final vertex of the face $d^ 0_ix$ to the final vertex of $x$, i.e., we have $d_1^0e^ 0_ix = d_1^1\cdots d_1^1d^0_ix$ and $d_1^1e^0_ix = d_1^ 1\cdots d_1^1x$. Similarly, $e^1_ix$ is an edge of $x$ leading from the initial vertex of $x$ to the initial vertex of the face $d^ 1_ix$, i.e., we have $d_1^ 0e^ 1_ix = {d_1^ 0\cdots d_1^0x}$ and $d_1^1e^ 1_ix = d_1^ 0\cdots d_1^0d_i^ 1x$.

\subsection{Free monoid} Let $\Sigma$ be an alphabet (i.e., a set). The free monoid on $\Sigma$ is denoted by $\Sigma^*$. The length of a string $m \in \Sigma^*$, i.e., the unique integer $n$ such that $m \in \Sigma^n$, will be denoted by $|m|$. Given a string $m$ of length $n\geq 1$, we will write $m_1, \dots, m_n$ to denote the uniquely determined elements of $\Sigma$ such that $m = m_1\cdots m_n$.

\subsection{Labeling cochain}

Let $\Sigma$ be an alphabet, and let $\A$ be a $\Sigma^*\mbox{-}$HDA. Consider the exterior algebra on the free module generated by $\Sigma$, $\Lambda (\Sigma)$, and let $n \geq 0$ be an integer. The \emph{$n$th labeling cochain} \[\mathfrak{l}_{\A}^n \colon C_*(\A) \to \Lambda (\Sigma)\] is the homomorphism of degree $-n$ defined on basis elements $x \in \A_n$ by 
\[\mathfrak{l}_{\A}^n(x) = \left \{\begin{array}{ll}
1, & n = 0, \vspace{0.2cm}\\
\sum \limits _{j_1 = 1}^{|\lambda_{\A}(e^0_1x)|} \dots \sum \limits _{j_n = 1}^{|\lambda_{\A}(e^0_nx)|} \lambda_{\A}(e^0_1x)_{j_1}\wedge \dots \wedge  \lambda_{\A}(e^0_nx)_{j_n},&  n > 0.
\end{array} \right.\]
Note that for an edge $x \in \A_1$ with $\lambda_{\A}(x) = m$, $\mathfrak{l}_{\A}^1(x) = \sum \limits _{j = 1}^{|m|} m_j$. Thus $\mathfrak{l}_{\A}^1(x)$ is the image of $\lambda_{\A}(x)$ under the canonical monoid homomorphism from $\Sigma^*$ to the free module on $\Sigma$. 

We show that the labeling cochain $\mathfrak{l}_{\A}^n$ is a cocycle. We need three lemmas.

\begin{lem} \label{hatlambdae}
For any element $x \in \A$ of degree $n\geq 1$, $\mathfrak{l}_{\A}^{1}(e^0_ix) = \mathfrak{l}_{\A}^{1}(e^1_ix)$.		
\end{lem}

\begin{proof}
If $n= 1$, then $e^0_ix = e^1_ix = x$. Suppose that $n\geq 2$. Then 
$e^0_ix= x_{\sharp}(1, \dots, 1, \underset{i}{[0,1]}, 1,\dots, 1)$ and $e^1_ix = x_{\sharp}(0, \dots, 0, \underset{i}{[0,1]}, 0,\dots, 0)$. Hence, by Lemma   \ref{paralleledges}, $\lambda_{\A}(e^0_ix) =  \lambda_{\A}(e^1_ix)$. Thus $\mathfrak{l}_{\A}^{1}(e^0_ix) = \mathfrak{l}_{\A}^{1}(e^1_ix)$.
\end{proof}

\begin{lem} \label{hatlambdaedges}
For any element $x \in \A$ of degree $n\geq 1$, 
\[\mathfrak{l}_{\A}^n(x) = \mathfrak{l}_{\A}^1(e^0_1x)\wedge \dots \wedge\mathfrak{l}_{\A}^1(e^0_nx) = \mathfrak{l}_{\A}^1(e^1_1x) \wedge \dots \wedge \mathfrak{l}_{\A}^1(e^1_nx).\]
\end{lem}

\begin{proof}
We have 
\[
\mathfrak{l}_{\A}^1(e^0_ix) = \sum \limits_{j_i=1}^{|\lambda_{\A}(e^0_ix)|} \lambda_{\A}(e^0_ix)_{j_i}
\]
and 
\begin{align*}
\mathfrak{l}_{\A}^n(x) 
&= \sum \limits _{j_1 = 1}^{|\lambda_{\A}(e^0_1x)|} \cdots \sum \limits _{j_n = 1}^{|\lambda_{\A}(e^0_nx)|}\lambda_{\A}(e^0_1x)_{j_1}\wedge \dots \wedge \lambda_{\A}(e^0_nx)_{j_n}\\
&= \left(\sum \limits_{j_1=1}^{|\lambda_{\A}(e^0_1x)|}  \lambda_{\A}(e^0_1x)_{j_1}\right) \wedge \dots \wedge \left(\sum \limits_{j_n=1}^{|\lambda_{\A}(e^0_nx)|} \lambda_{\A}(e^0_nx)_{j_n}\right)\\
&= \mathfrak{l}_{\A}^1(e^0_1x) \wedge \dots \wedge \mathfrak{l}_{\A}^1(e^0_nx).
\end{align*}
By Lemma \ref{hatlambdae}, also $\mathfrak{l}_{\A}^n(x) =  \mathfrak{l}_{\A}^1(e^1_1x) \wedge \dots \wedge \mathfrak{l}_{\A}^1(e^1_nx)$.		
\end{proof}

\begin{lem} \label{hatlambdabdy}
For all $n\geq 0$, $x \in \A_{n+1}$, and $i\in \{1, \dots, n+1\}$, \[\mathfrak{l}_{\A}^{n}(d^0_ix) = \mathfrak{l}_{\A}^{n}(d^1_ix).\]	
\end{lem}

\begin{proof}
If $n=0$, then $\mathfrak{l}_{\A}^{n}(d^0_ix) = 1 = \mathfrak{l}_{\A}^{n}(d^1_ix)$. If $n=1$, then $\lambda_{\A}(d^0_ix) = \lambda_{\A}(d^1_ix)$ and therefore $\mathfrak{l}_{\A}^{1}(d^0_ix) = \mathfrak{l}_{\A}^{1}(d^1_ix)$. Suppose that $n \geq 2$. We have 
\[ e^{1-k}_jd^k_ix = d_1^{k} \cdots d_{j-1}^ {k}d_{j+1}^{k}\cdots d_n^{k}d^k_ix = \left \{ \begin{array}{ll} e^{1-k}_jx, & i > j,\vspace{0.2cm} \\
e^{1-k}_{j+1}x, &  i\leq  j. \end{array} \right. \]
Hence, by Lemma \ref{hatlambdaedges} and Lemma \ref{hatlambdae}, \begin{align*}
\mathfrak{l}_{\A}^{n}(d^0_ix) &= \mathfrak{l}_{\A}^1(e^{1}_1d^0_ix) \wedge \dots \wedge \mathfrak{l}_{\A}^1(e^{1}_nd^0_ix)\\ &= \mathfrak{l}_{\A}^1(e^{1}_1x) \wedge \dots \wedge \mathfrak{l}_{\A}^1(e^{1}_{i-1}x) \wedge  \mathfrak{l}_{\A}^1(e^{1}_{i+1}x)\wedge \dots \wedge \mathfrak{l}_{\A}^1(e^{1}_{n+1}x)\\
&= \mathfrak{l}_{\A}^1(e^{0}_1x) \wedge \dots \wedge \mathfrak{l}_{\A}^1(e^{0}_{i-1}x) \wedge \mathfrak{l}_{\A}^1(e^{0}_{i+1}x) \wedge \dots \wedge \mathfrak{l}_{\A}^1(e^{0}_{n+1}x)\\
&= \mathfrak{l}_{\A}^1(e^{0}_1d^1_ix) \wedge \dots \wedge \mathfrak{l}_{\A}^1(e^{0}_nd^1_ix)\\
&= \mathfrak{l}_{\A}^{n}(d^1_ix).\qedhere
\end{align*}
\end{proof}

\begin{prop} \label{cocycle}
For each $n\geq 0$, $\mathfrak{l}_{\A}^n$ is a cocycle.	
\end{prop}

\begin{proof}
Let $x \in \A_{n+1}$. By Lemma \ref{hatlambdabdy}, \[\mathfrak{l}_{\A}^n (dx) =  \sum \limits_{i=1}^{n+1}(-1)^i(\mathfrak{l}_{\A}^{n}(d^0_ix) -\mathfrak{l}_{\A}^{n}(d^1_ix)) = 0.\qedhere\]
\end{proof}

\subsection{Labeled homology} 
Let $\Sigma$ be an alphabet, and let $\A$ be a $\Sigma^*\mbox{-}$HDA. Since the labeling cochain $\mathfrak{l}_{\A}^n$ is a cocycle, we may define the \emph{$n$th labeling homomorphism} \[ \ell_{\A}^n\colon H_n(\A) \to\Lambda (\Sigma)\]
by \[\ell_{\A}^n([z])  = {\mathfrak{l}}_{\A}^n(z).\]
We will be able to drop the superscript $n$ and write $\ell_{\A}$ instead of $\ell_{\A}^n$. The \emph{labeled homology} of $\A$ is the homology of $\A$ together with the sequence of labeling homomorphisms.

\begin{exs} \label{firstexs}
	\begin{sloppypar}
(i) Consider the alphabet $\Sigma = \{a,b\}$ and the 1-dimensional $\Sigma^*\mbox{-}$HDA $\A$ made up of the boundary of a square in the way depicted in Figure \ref{figbd}.
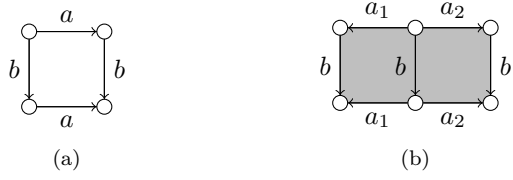
\begin{figure}[t]
	\center
	\subfloat[]
	{ 
	\begin{tikzpicture}[initial text={},on grid]

	\path[draw] (0,0)--(1,0)--(1,1)--(0,1)--cycle;

	\node[state,minimum size=0pt,inner sep =2pt,fill=white] (p_0)   {}; 
	
	\node[state,minimum size=0pt,inner sep =2pt,fill=white] (p_2) [right=of p_0,xshift=0cm] {};
	
	\node[state,minimum size=0pt,inner sep =2pt,fill=white,initial where=above,initial distance=0.2cm] [above=of p_0, yshift=0cm] (p_3)   {};

	\node[state,minimum size=0pt,inner sep =2pt,fill=white] (p_5) [right=of p_3,xshift=0cm] {}; 
	
	\path[->] 
	(p_0) edge[below] node {$a$} (p_2)
	(p_3) edge[above]  node {$a$} (p_5)
	(p_3) edge[left]  node {$b$} (p_0)
	(p_5) edge[right]  node {$b$} (p_2);

	\end{tikzpicture}
	 \label{figbd}
}
\hspace{2cm}
\subfloat[]
{
	\begin{tikzpicture}[initial text={},on grid]

	\path[draw, fill=lightgray] (0,0)--(1,0)--(1,1)--(0,1)--cycle;     
	
	\path[draw, fill=lightgray]
	(0,0)--(-1,0)--(-1,1)--(0,1)--cycle;

	\node[state,minimum size=0pt,inner sep =2pt,fill=white] (q_0)   {}; 
	
	\node[state,minimum size=0pt,inner sep =2pt,fill=white] (q_2) [right=of q_0,xshift=0cm] {};
	
	\node[state,minimum size=0pt,inner sep =2pt,fill=white] (q_4) [left=of q_0,xshift=0cm] {};
	
	\node[state,minimum size=0pt,inner sep =2pt,fill=white] [above=of q_0, yshift=0cm] (q_3)   {};
	
	\node[state,minimum size=0pt,inner sep =2pt,fill=white] (q_6) [left=of q_3,xshift=0cm] {};

	\node[state,minimum size=0pt,inner sep =2pt,fill=white] (q_5) [right=of q_3,xshift=0cm] {}; 
	
	\path[->] 
	(q_0) edge[below] node {$a_2$} (q_2)
	(q_0) edge[below] node {$a_1$} (q_4)
	(q_3) edge[above]  node {$a_2$} (q_5)
	(q_3) edge[above]  node {$a_1$} (q_6)
	(q_6) edge[left]  node {$b$} (q_4)
	(q_3) edge[left]  node {$b$} (q_0)
	(q_5) edge[right]  node {$b$} (q_2);

	\end{tikzpicture}
	 \label{figKlein}
}
\caption{Two very simple HDAs}
\end{figure}  
For each edge $x$, we have $\mathfrak{l}_{\A}^1(x) = \lambda_{\A}(x)$. Therefore $\ell_{\A}(\alpha) = 0$ for all $\alpha \in H_1(\A)$. This reflects the fact that every $1$-cycle would be a boundary if the square was not missing in $\A$. A hole in an HDA that is due to a set of missing cubes is related to what is called a \emph{forbidden region} in \cite{FGHMR}. The label of the homology class corresponding to such a hole is always zero. 
\end{sloppypar}
(ii) Since $\Lambda (\Sigma)$ is a free module, the label of any torsion element in the homology of a $\Sigma^ *\mbox{-}$HDA is $0$.

(iii) A 1-dimensional $\Sigma^*\mbox{-}$HDA $\A$ is called a \emph{directed circle} if it has $n$ distinct vertices $v_0, \dots, v_{n-1}$ $(n\geq 1)$ and $n$ edges $x_0, \dots, x_{n-1}$ such that for all $i$ and $k$, $d_1^kx_i = v_{i+k}$ (where $v_n = v_0$). 
The chain $z=\sum_{i=0}^{n-1}x_i$ is a cycle, and if $\lambda_{\A}(x_i) = a_{i1}\cdots a_{ir_i}$ with $a_{i1},\dots ,a_{ir_i}\in \Sigma$, then \[\ell_{\A}([z]) = \sum \limits_{i=0}^{n-1} \sum\limits_{j=1}^{r_i} a_{ij}.\]   

(iv) Consider the alphabet $\Sigma = \{a_1,a_2,b\}$ and the  $\Sigma^*\mbox{-}$HDA $\A$ obtained from the one depicted in Figure \ref{figKlein} by identifying the horizontal rows and the left and right vertical edges. Then $\A$ is a torus and its $\Z_2$-homology is generated by two 1-dimensional classes with labels $a_1+a_2$ and $b$ and by a 2-dimensional class with label $a_1\wedge b + a_2\wedge b$.

If $a_1 = a_2$, we may identify the horizontal rows of the HDA in Figure \ref{figKlein} not only to obtain a torus but also in a way that yields the Klein bottle. In either case, the $\Z_2$-homology of the resulting HDA is generated by two 1-dimensional classes with labels $0$ and $b$ and by a 2-dimensional class with label $0$.

\end{exs}

\section{Tensor-product HDAs} \label{secTensor}

The tensor product of HDAs models the interleaving parallel composition of concurrent systems (see \cite{transhda}). In this section we show that the labeling cochains and the labeling homomorphisms of a tensor-product HDA are determined by those of its  factors.

\subsection{Tensor product of HDAs} The tensor product of two $M\mbox{-}$HDAs $\A$ and $\B$ is the $M\mbox{-}$HDA $\A\otimes \B$ where $P_{\A \otimes \B} = P_{\A}\otimes P_{\B}$, $I_{\A \otimes \B} = I_{\A}\times I_{\B}$, $F_{\A \otimes \B} = F_{\A}\times F_{\B}$, and $\lambda_{\A \otimes \B}$ is the composite
\[(\A_1\times \B_0) \amalg (\A_0\times \B_1) \to \A_1\amalg \B_1 \xrightarrow{(\lambda_{\A},\lambda_{\B})} M.\]
With respect to this tensor product, the category of $M\mbox{-}$HDAs is a monoidal category.

\subsection{The maps \texorpdfstring{$\zeta$}{} and \texorpdfstring{$\gamma$}{} and the cross products} \label{cross} Given (co)chain complexes $C$ and $D$, let $\kappa$ denote the natural homomorphism of graded modules 
\[H(C)\otimes H(D) \to H(C\otimes D), \quad [x]\otimes [y] \mapsto [x\otimes y].\] 
Let $P$ and $Q$ be precubical sets. The \emph{homology cross product} 
\[ \times \colon  H_*(P) \otimes H_*(Q) \to  H_*(P\otimes Q)\] is defined to be the composite $\zeta_*\circ \kappa$ where $\zeta$ 
is the natural isomorphism of chain complexes \[ C_*(P)\otimes C_*(Q) \to C_*(P\otimes Q)\] given by \[x\otimes y \mapsto (x,y), \quad x\in P,\, y\in Q.\]

Let $A$ and $B$ be modules, and consider the cochain complexes $C^*(P;A) = \Hom (C_*(P),A)$ and $C^*(Q;B) = \Hom (C_*(Q),B)$. The \emph{cohomology cross product} 
\[ \times \colon  H^*(P;A) \otimes H^*(Q;B) \to  H^*(P\otimes Q;A\otimes B)\] is defined to be the composite $H(\Hom(\zeta^{-1},A\otimes B)\circ \gamma) \circ \kappa$ where $\gamma$ is the natural chain map 
\[ \Hom (C_*(P),A) \otimes  \Hom (C_*(Q),B) \to  \Hom (C_*(P)\otimes C_*(Q),A \otimes B) \]     
given by 
\[\gamma(f\otimes g)(x\otimes y) = (-1)^{\deg(g)\deg(x)}f(x) \otimes g(y). \]

The homology cross product for precubical sets coincides with the usual homology cross product for spaces in the sense that it is isomorphic, as one easily checks, to the composite
\[H_*(|P|) \otimes H_*(|Q|) \xrightarrow{\kappa} H(S_*(|P|)\otimes S_*(|Q|)) \xrightarrow{EZ_*} H_*(|P|\times |Q|)\]
where $EZ$ is the \emph{Eilenberg-Zilber map}
\[S_*(|P|) \otimes S_*(|Q|) \to S_*(|P|\times |Q|),\] which for singular cubes $f\colon [0,1]^p \to |P|$ and $g\colon [0,1]^q \to |Q|$ is given by 
\[(f+D_p(S|P|))\otimes (g+D_q(S|Q|)) \mapsto f\times g + D_{p+q}(S(|P|\times |Q|))\] 
(see \cite[VI.3]{Massey}, \cite[8.7.5]{HiltonWylie}). The cohomology cross products for precubical sets and spaces are related in a similar fashion.

\subsection{Labeled homology of tensor-product HDAs} Let $\Sigma$ be an alphabet, and let $\A$ and $\B$ be two $\Sigma^*\mbox{-}$HDAs. 

\begin{lem} \label{hatlambdaotimes}
	Consider elements $x\in \A_p$ and $y \in \B_q$. Then 
	\[\mathfrak{l}^{p+q}_{\A \otimes \B}(x,y) = \mathfrak{l}_{\A}^p(x) \wedge \mathfrak{l}_{\B}^q(y).\]
\end{lem}

\begin{proof}
We may suppose that $p+q > 0$. We have 
\[e^0_i(x,y) = \left \{ \begin{array}{ll}
\vspace{0.2cm}
(e^0_ix,\underbracket[0.5pt]{d^1_1\cdots d^1_1}_{q\; {\text{times}}}y), & i\leq p, \\
(\underbracket[0.5pt]{d^1_1\cdots d^1_1}_{p\; {\text{times}}}x,e^0_{i-p}y) & i> p.\end{array} \right.\]
Hence \[\lambda_{\A \otimes \B}(e^0_i(x,y)) = \left \{ \begin{array}{ll} \lambda_{\A}(e^0_ix), & i\leq p, \\ \lambda_{\B}(e^0_{i-p}y) & i> p.\end{array} \right.\]
Therefore
\[\mathfrak{l}_{\A \otimes \B}^{1}(e^0_i(x,y)) = \left \{ \begin{array}{ll} \mathfrak{l}_{\A}^1(e^0_ix), & i\leq p, \\ \mathfrak{l}_{\B}^1(e^0_{i-p}y) & i> p.\end{array} \right.\]
By Lemma \ref{hatlambdaedges}, it follows that \begin{align*}\mathfrak{l}_{\A\otimes \B}^{p+q}(x,y) &= \mathfrak{l}_{\A \otimes \B}^1(e^0_1(x,y))\wedge  \cdots \wedge \mathfrak{l}_{\A \otimes \B}^1(e^0_{p+q}(x,y))\\&= \mathfrak{l}_{\A}^1(e^0_1x) \wedge \cdots \wedge \mathfrak{l}_{\A}^1(e^0_px) \wedge   \mathfrak{l}_{\B}^1(e^0_{1}y) \wedge \cdots \wedge \mathfrak{l}_{\B}^1(e^0_{q}y)\\  &=  \mathfrak{l}_{\A}^p(x) \wedge \mathfrak{l}_{\B}^q(y). \qedhere 
\end{align*} 
\end{proof}

Using the multiplication of $\Lambda(\Sigma)$ and the maps $\zeta$ and $\gamma$ defined in Section \ref{cross}, we are now able to express the labeling cochains of $\A\otimes \B$ in terms of the labeling cochains of $\A$ and $\B$:

\begin{prop} \label{crosslem}
For each $n \geq 0$, ${\mathfrak{l}_{\A \otimes \B}^{n}} = \sum \limits_{i+j=n} {(-1)^{ij} \wedge \circ \gamma(\mathfrak{l}_{\A}^{i}\otimes \mathfrak{l}_{\B}^{j})\circ \zeta^{-1}}$.
\end{prop}

\begin{proof}
Let $n\geq 0$. Consider elements ${x \in \A_p}$ and ${y \in \B_q}$ where $p+q = n$. By Lemma \ref{hatlambdaotimes}, 
\begin{align*}
{(-1)^{pq} \wedge \circ \gamma(\mathfrak{l}_{\A}^{p}\otimes \mathfrak{l}_{\B}^{q})\circ \zeta^{-1}(x,y)} &= {(-1)^{pq} \wedge \circ \gamma(\mathfrak{l}_{\A}^{p}\otimes \mathfrak{l}_{\B}^{q})(x\otimes y)}\\ &= {(-1)^{pq} \wedge((-1)^{p(-q)}\mathfrak{l}_{\A}^p(x) \otimes \mathfrak{l}_{\B}^q(y))}\\ &= {\mathfrak{l}_{\A}^p(x) \wedge \mathfrak{l}_{\B}^q(y)}\\
&=  \mathfrak{l}_{\A \otimes \B}^{n} (x,y).
\end{align*}
For $p'\not=p$ and $q'\not=q$ with $p'+q' = n$,
\[(-1)^{p'q'} \wedge \circ \gamma(\mathfrak{l}_{\A}^{p'}\otimes \mathfrak{l}_{\B}^{q'})\circ \zeta^{-1}(x,y) = \pm \, {\mathfrak{l}_{\A}^{p'}(x) \wedge \mathfrak{l}_{\B}^{q'}(y)} = 0.\]
Hence ${\mathfrak{l}_{\A \otimes \B}^{n}(x,y)} = \sum \limits_{i+j=n} {(-1)^{ij} \wedge \circ \gamma(\mathfrak{l}_{\A}^{i}\otimes \mathfrak{l}_{\B}^{j})\circ \zeta^{-1}(x,y)}$.
\end{proof}

Let us write $\wedge_*$ to denote the map \[H^*(\A\otimes \B,\Lambda (\Sigma)\otimes \Lambda(\Sigma)) \to H^*(\A\otimes \B,\Lambda (\Sigma))\] induced by the multiplication of $\Lambda (\Sigma)$.

\begin{theor} \label{cocrosslabel}
For all $n \geq 0$, $[\mathfrak{l}_{\A \otimes \B}^{n}] = \wedge_*(\sum\limits_{i+j=n} (-1)^{ij} [\mathfrak{l}_{\A}^{i}] \times [\mathfrak{l}_{\B}^{j}])$.
\end{theor}

\begin{proof}
This follows immediately from Proposition \ref{crosslem} because  \[{\wedge_*([\mathfrak{l}_{\A}^{i}] \times [\mathfrak{l}_{\B}^{j}])} = {[\wedge\circ \gamma(\mathfrak{l}_{\A}^{i}\otimes \mathfrak{l}_{\B}^{j})\circ \zeta^{-1}]}. \qedhere\]
\end{proof}

We shall show next that the label of the cross product of two homology classes is the exterior product of their labels. This implies that the labeling homomorphisms of $\A \otimes \B$ are entirely determined by those of $\A$ and $\B$. Indeed, by the Künneth theorem, there exists a graded torsion module $U\subseteq H_*(\A\otimes \B)$ such that \[{H_*(\A\otimes \B)} = U \oplus \im \;H_*(\A)\otimes H_*(\B)\xrightarrow{\times} H_*(\A\otimes \B),\] and, as we have noted in Example \ref{firstexs}(ii), torsion elements have zero label.

\begin{theor} \label{crosslabel} 
For all homology classes $\alpha \in H_p(\A)$ and $\beta \in H_q(\B)$, \[\ell_{\A \otimes \B}(\alpha \times \beta) = \ell_{\A}(\alpha)\wedge \ell_{\B}(\beta).\]
\end{theor}

\begin{proof}
Let ${a\in C_p(\A)}$ and ${b\in C_q(\B)}$ be cycles such that ${\alpha = [a]}$ and ${\beta = [b]}$. By Proposition  \ref{crosslem}, we have 
\begin{align*}
{\ell_{\A \otimes \B}(\alpha \times \beta)} &= {\ell_{\A \otimes \B}([a]\times [b])}\\ 
&= {\ell_{\A \otimes \B}([\zeta(a \otimes b)])}\\ 
&= {\mathfrak{l}_{\A \otimes \B}^{p+q}(\zeta(a \otimes b))}\\ 
&=
\sum \limits_{i+j=p+q}
 {(-1)^{ij}\wedge\circ \gamma(\mathfrak{l}_{\A}^{i}\otimes \mathfrak{l}_{\B}^{j})\circ \zeta^{-1}(\zeta(a\otimes b))}\\ 
 &=
 \sum \limits_{i+j=p+q}
 {(-1)^{ij} (-1)^{p(-j)}\mathfrak{l}_{\A}^{i}(a)\wedge \mathfrak{l}_{\B}^{j}(b)}\\
 &= {(-1)^{pq} (-1)^{pq}\mathfrak{l}_{\A}^{p}(a)\wedge \mathfrak{l}_{\B}^{q}(b)}\\ 
 &= {\mathfrak{l}_{\A}^{p}(a)\wedge \mathfrak{l}_{\B}^{q}(b)}\\   
 &=  {\ell_{\A}(\alpha)\wedge \ell_{\B}(\beta)}. 
 \qedhere   
\end{align*}
\end{proof}

\begin{ex}
A \emph{directed torus} is a $\Sigma^*\mbox{-}$HDA of the form $\A \otimes \B$ where $\A$ and $\B$ are directed circles. Let $x_0, \dots, x_{n-1}$ and $y_0, \dots, y_{m-1}$ be the edges of $\A$ and $\B$, respectively, and consider the cycles $x = \sum_i x_i$ and $y = \sum_jy_j$. Then the chain $z = \sum_{i,j}(x_i,y_j)$ is a cycle in $C_*(\A\otimes \B)$ and $[z] = [x]\times [y]$. Suppose that $\lambda_{\A}(x_i)  = a_{i1}\cdots a_{ir_i}$ and $\lambda_{\B}(y_j) =  b_{j1}\cdots b_{js_j}$ with all $a_{ip}, b_{jq} \in \Sigma$. Then \[\ell_{\A\otimes \B}([z]) = {\ell_{\A}([x])\wedge \ell_{\B}([y])} =\sum _{i,p,j,q} a_{ip}\wedge b_{jq}.\] 
\end{ex}

\section{Cubical dimaps} \label{secDimaps} 
The purpose of this section is to introduce cubical dimaps. Informally, a cubical dimap between two HDAs is a continuous map between their geometric realizations that sends cubes in an order-preserving way to subdivided cubes and that preserves labels of paths. Cubical dimaps are more flexible than morphisms of HDAs and permit one to relate HDAs that model concurrent systems at different refinement levels.

\subsection{Permutation maps} 
Given $n$ precubical sets
$P_1, \dots, P_n$ and a permutation $\sigma \in S_n$, we denote by $t_{\sigma}$ the \emph{permutation map} 
\[|P_1\otimes \dots \otimes P_n| = |P_1|\times \dots \times |P_n| \to |P_{\sigma(1)}\otimes \dots \otimes P_{\sigma(n)}| = |P_{\sigma(1)}|\times \dots \times |P_{\sigma(n)}|\]
given by 
\[(a_1, \dots, a_n)\mapsto  (a_{\sigma(1)}, \dots , a_{\sigma(n)}).\]
We remark that for permutations $\sigma, \tau \in S_n$, $t_{\sigma \circ \tau} = t_{\tau} \circ t_{\sigma}$. If $n=2$ and $\sigma$ is the transposition $(1\;2)$, we write $t$ instead of $t_{\sigma}$.
\subsection{Cubical dimaps of precubical sets}

An \emph{elementary cubical dimap (directed map)} from a precubical set $Q$ to a precubical set $P$ is a continuous map $f\colon |Q| \to |P|$ such that the following two conditions hold:
\begin{enumerate}
	\item For every vertex ${v\in Q_0}$, there exists a (necessarily unique) vertex $w\in P_0$ such that ${f([v,()])} = {[w,()]}$.
	\item For every element ${x\in Q_n}$ ${(n>0)}$, there exist integers $l_1, \dots, l_n\geq 1$, a morphism of precubical sets ${\chi \colon \llbracket 0,{l_1} \rrbracket\otimes \cdots \otimes \llbracket 0,{l_n} \rrbracket} \to P$, a permutation $\sigma \in S_n$, and increasing homeomorphisms ${\phi_i \colon |\llbracket 0,1 \rrbracket| = [0,1] \to |\llbracket 0,{l_i} \rrbracket|} =  [0,l_i ]$ $(i \in {\{1, \dots, n\})}$  such that ${f\circ |x_{\sharp}|} = {|\chi| \circ (\phi_1 \times \cdots \times \phi_n)\circ t_{\sigma}}$.
\end{enumerate}
A \emph{cubical dimap} of precubical sets is a finite composite of elementary cubical dimaps. 

\begin{exs} \label{dimapexs}
(i) The geometric realization of a morphism of precubical sets is an elementary cubical dimap.

(ii) An important class of cubical dimaps is given by subdivisions in the sense of \cite{weakmor}: A subdivision of a precubical set $Q$ consists of a precubical set $P$ and a homeomorphism $|Q| \to |P|$ that is an elementary cubical dimap such that the permutation in condition (2) of the definition is always the identity.

(iii) For precubical sets $P_1, \dots, P_n$ and a permutation $\sigma \in S_n$, the permutation map
$t_{\sigma}\colon |P_1\otimes \dots \otimes P_n| \to |P_{\sigma(1)}\otimes \dots \otimes P_{\sigma(n)}|$
is an elementary cubical dimap. As we do not need this fact in this paper, we omit the details.

\end{exs}

\begin{rems}
(i) Condition (2) in the definition of elementary cubical dimaps guarantees that cubical dimaps are directed maps in the sense that they respect the directed topology of geometric realizations of precubical sets. A prominent framework for directed topology is Grandis' category of d-spaces \cite{GrandisBook}. A cubical dimap can be seen as a d-map, i.e., a morphism of d-spaces.

(ii) Cubical dimaps are closely related to weak morphisms as studied in \cite{hgraph, weakmor, topabs}. Although the two concepts are formally incomparable, a weak morphism will normally be a cubical dimap. Permutation maps are cubical dimaps but not in general weak morphisms. 

(iii) By construction, a cubical dimap is a cellular map. 
\end{rems}

We next show that the objects in condition (2) of the definition of elementary cubical dimaps are unique. 

\begin{sloppypar}
\begin{lem} {\rm \cite[2.3.3]{weakmor}} \label{lemuniqueness}
	Consider integers ${n, k_1, \dots, k_n,l_1, \dots, l_n\geq 1}$, morphisms of precubical sets ${\xi \colon \llbracket 0,{k_1} \rrbracket\otimes \cdots \otimes \llbracket 0,{k_n} \rrbracket \to P}$ and $\zeta \colon {\llbracket 0,{l_1} \rrbracket\otimes \cdots \otimes \llbracket 0,{l_n} \rrbracket} \to P$, and a homeomorphism $\alpha \colon {|\llbracket 0,{k_1} \rrbracket\otimes \cdots \otimes \llbracket 0,{k_n} \rrbracket|} \to {|\llbracket 0,{l_1} \rrbracket\otimes \cdots \otimes \llbracket 0,{l_n} \rrbracket|}$ such that ${|\zeta| \circ \alpha}  = |\xi|$. Then $k_i = l_i$ for all $i \in \{1, \dots,n\}$, $\alpha = id$, and $\xi = \zeta$.  
\end{lem}

\begin{prop} \label{elcontuniqueness} 
	Let $f\colon |Q| \to |P|$ be an elementary cubical dimap of precubical sets, and let $x \in Q_n$ $(n \geq 1)$ be an element. Then there exist unique integers $l_1, \dots, l_n\geq 1$, a unique morphism of precubical sets $\chi \colon {\llbracket 0,{l_1} \rrbracket\otimes \cdots \otimes \llbracket 0,{l_n} \rrbracket} \to P$, a unique permutation $\sigma \in S_n$, and unique increasing homeomorphisms $\phi_i \colon {|\llbracket 0,1 \rrbracket|} = [0,1] \to {|\llbracket 0,{l_i} \rrbracket|} =  [0,l_i ]$ $(i \in \{1, \dots, n\})$  such that ${f\circ |x_{\sharp}|} = {|\chi| \circ (\phi_1 \times \cdots \times \phi_n)\circ t_{\sigma}}$.  	
\end{prop}

\begin{proof}
	The existence of $l_1, \dots, l_n$, $\chi $, $\sigma$, and $\phi_1, \dots, \phi_n$ is guaranteed by the definition of elementary cubical dimaps. Suppose that the integers $l'_1, \dots, l'_n$, the morphism of precubical sets $\chi'  \colon \llbracket 0,{l'_1} \rrbracket \otimes \cdots \otimes \llbracket 0,{l'_n} \rrbracket \to P$, the permutation ${\sigma' \in S_n}$, and the increasing homeomorphisms $\phi'_i \colon  |\llbracket 0,{1} \rrbracket|  \to |\llbracket 0,{l'_i} \rrbracket|$ $(i \in \{1, \dots, n\})$ satisfy ${f\circ |x_{\sharp}|} = {|\chi' |\circ (\phi'_1 \times \cdots \times \phi'_n)\circ t_{\sigma'}}$. Consider the homeomorphisms $\phi = {(\phi_1 \times \cdots \times \phi_n)\circ t_{\sigma}}$, $\psi = {(\phi'_1 \times \cdots \times \phi'_n)\circ t_{\sigma'}}$, and $\alpha = {\psi \circ \phi^{-1}} \colon {|\llbracket 0,{l_1} \rrbracket \otimes \cdots \otimes \llbracket 0,{l_n} \rrbracket|} \to {|\llbracket 0,{l'_1} \rrbracket \otimes \cdots \otimes \llbracket 0,{l'_n} \rrbracket|}$. Then $|\chi'| \circ \alpha = {|\chi'|\circ \psi \circ \phi^{-1}} = f\circ |x_{\sharp} | \circ \phi^{-1} = |\chi| \circ \phi \circ \phi^ {-1} = |\chi|$. By Lemma \ref{lemuniqueness}, it follows that $l'_i = l_i$ for all $i \in \{1, \dots,n\}$, $\psi = \phi$, and $\chi' = \chi$. We show that $\sigma' = \sigma$. This then also implies that $\phi'_1\times \dots \times \phi'_n = \phi_1\times \dots \times \phi_n$ and hence that $\phi'_i = \phi_i$ for all $i$. Consider $i \in \{1, \dots, n\}$, and set \[x_j = \left\{ \begin{array}{ll}0, & j \not= \sigma(i),\\ \frac{1}{2}, & j = \sigma(i). \end{array} \right.\] We have 
	\begin{eqnarray*}
		(\phi_1\times \dots \times \phi_n)\circ t_{\sigma}(x_1, \dots ,x_n) &=& (\phi_1(x_{\sigma(1)}), \dots ,\phi_i(x_{\sigma(i)}), \dots , \phi_n(x_{\sigma(n)}))\\
		&=& (0, \dots, 0, \phi_i(\tfrac{1}{2}), 0, \dots ,0).
	\end{eqnarray*}
	On the other hand, 
	\[(\phi'_1\times \dots \times \phi'_n)\circ t_{\sigma'}(x_1, \dots ,x_n) = (\phi'_1(x_{\sigma'(1)}), \dots ,\phi'_n(x_{\sigma'(n)})).\]
	Hence \[\phi'_j(x_{\sigma'(j)}) = \left\{ \begin{array}{ll}0, & j \not= i,\\ \phi_i(\frac{1}{2}), & j = i. \end{array} \right.\]
	Since $\phi_i$ and the $\phi'_j$ are increasing homeomorphisms, this implies that 
	\[x_{\sigma'(j)} \not= 0 \Leftrightarrow j=i.\] Thus $\sigma'(i) = \sigma(i)$. 		
\end{proof}
\end{sloppypar}

\subsection{Notation}
Given precubical sets $P$ and $Q$ and an elementary cubical dimap $f\colon |Q| \to |P|$, Proposition \ref{elcontuniqueness} permits us to fix the following notation for the objects in condition (2) of the definition of elementary cubical dimaps: We write $R_x = \llbracket 0,{l_1} \rrbracket\otimes \cdots \otimes \llbracket 0,{l_n} \rrbracket$, $\phi_{xi} = \phi_i$, $\sigma_x = \sigma$, and $x_{\flat} = \chi$. Note that here  we silently omit any reference to $f$.

\subsection{Vertex map} It is clear that condition (1) in the definition of elementary cubical dimaps holds for arbitrary cubical dimaps as well. Therefore a cubical dimap $f\colon |Q| \to |P|$ induces a \emph{vertex map} $f_0\colon Q_0 \to P_0$, which sends a vertex $v \in Q_0$ to the unique vertex $w \in P_0$ such that $f([v,()]) = [w,()]$.

\subsection{Paths} 
Let  $f\colon |Q| \to |P|$ be a cubical dimap, and let $\omega \colon \llbracket 0,k\rrbracket \to Q$ be a path.

\begin{prop} \label{fpath}
	 There exist a unique integer $l$, a unique  path $\nu \colon \llbracket 0, l\rrbracket \to P$, and a unique increasing homeomorphism $\phi \colon |\llbracket 0,k\rrbracket| = [0,k] \to  |\llbracket 0,l\rrbracket| = [0,l]$ such that $|\nu| \circ \phi = f \circ |\omega|$. 
\end{prop}

\begin{proof}
	If $k=0$, we may and must set $l= 0$, $\phi = id$, and $\nu = f_0(\omega(0))_{\sharp}$. 
	
	Suppose that $k>0$. Write $\omega = x_{1\sharp} \cdots x_{k\sharp}$. To show existence, suppose first that $f$ is elementary. For each $i$, let $R_{x_i} = \llbracket 0,l_i\rrbracket$ and write $\phi_i = \phi_{x_i1}$. Since $\sigma_{x_i} = id$, we have $f\circ |x_{i\sharp}| = |x_{i\flat}|\circ \phi_i$. In order to define $\nu$, we will concatenate 
	the paths $x_{i\flat}$. This is possible because for $1 \leq i < k$, $x_{i\flat}(l_i) = x_{(i+1)\flat}(0)$. Indeed,  
	\begin{align*}
	[x_{i\flat}(l_i),()] &= |x_{i\flat}|([l_i,()])
	= |x_{i\flat}|\circ \phi_i([1,()])
	= f\circ |x_{i\sharp}|([1,()])
	= f([x_{i\sharp}(1),()])\\
	&= f([x_{(i+1)\sharp}(0),()]) 
	= f\circ |x_{(i+1)\sharp}|([0,()])
	= |x_{(i+1)\flat}|\circ \phi_{i+1} ([0,()])\\
	&= |x_{(i+1)\flat}|([0,()])
	= [x_{(i+1)\flat}(0),()].
	\end{align*}	
	 Set $l = l_1+ \dots + l_k$ and $\nu = x_{1\flat} \cdots x_{k\flat}$, and consider the increasing homeomorphism $\phi \colon [0,k] \to [0,l]$ given by 
	\[\phi(t) = l_1 + \dots + l_{i-1} + \phi_{i}(t-i+1), \quad t \in [i-1,i], \, i \in \{1, \dots, k\}.\]
	For $i \in \{1, \dots, k\}$ and  $t\in [i-1,i]$, we have 
	\begin{align*}
	|\nu|\circ \phi (t) 
	&= |\nu| (l_1 + \dots + l_{i-1}+\phi_i(t-i+1))\\
	&= |x_{i\flat}|\circ \phi_i(t-i+1)\\ 
	&= f\circ |x_{i\sharp}|(t-i+1)\\ 
	&=  f\circ |\omega|(t).
	\end{align*}
	Hence $|\nu|\circ \phi = f\circ |\omega|$, and existence is shown for elementary cubical dimaps. This implies that $l$, $\nu$, and $\phi$  exist if $f$ is an arbitrary cubical dimap. Consider an integer $l'$, a path $\nu' \colon \llbracket 0, l'\rrbracket \to P$, and an increasing homeomorphism $\psi \colon |\llbracket 0,k\rrbracket| = [0,k] \to  |\llbracket 0,l'\rrbracket| = [0,l']$ such that $|\nu'| \circ \psi = f \circ |\omega|$. Consider the homeomorphism $\alpha = \psi \circ \phi^{-1} \colon |\llbracket 0,{l} \rrbracket| \to |\llbracket 0,{l'} \rrbracket|$. Then $|\nu'| \circ \alpha = |\nu'|\circ \psi \circ \phi^{-1} = f\circ |\omega | \circ \phi^{-1} = |\nu| \circ \phi \circ \phi^ {-1} = |\nu|$. By Lemma \ref{lemuniqueness}, it follows that $l' = l$, $\psi = \phi$, and $\nu' = \nu$.
\end{proof}

The path $\nu$ of Proposition \ref{fpath} will be denoted by $f^{\mathbb I}(\omega)$. We remark that if $k= 0$, $f^{\mathbb I}(\omega)$ is the path in $P$ of length $0$ given by $f^{\mathbb I}(\omega)(0) = f_0(\omega(0))$. Note also that if $f$ is the geometric realization of a morphism of precubical sets $h\colon Q \to P$, then $f^{\mathbb I}(\omega) = h\circ \omega$. Note finally that if $g\colon |P| \to |K|$ is a second cubical dimap, then $(g\circ f)^{\mathbb I}(\omega) = g^{\mathbb I} \circ f^{\mathbb I}(\omega)$.

\subsection{Cubical dimaps of HDAs}	
	An \emph{elementary cubical dimap} from an $M\mbox{-}$HDA $\B$ to an $M\mbox{-}$HDA $\A$ is an elementary cubical dimap of precubical sets  $f\colon |\B| \to |\A|$  such that $f_0(I_{\B}) \subseteq I_{\A}$, $f_0(F_{\B}) \subseteq F_{\A}$, and  $\overline{\lambda}_{\A}\circ f^{\mathbb I} = \overline{\lambda}_{\B}$. A \emph{cubical dimap} of $M\mbox{-}$HDAs is a finite composite  of elementary cubical dimaps. It is clear that the conditions $f_0(I_{\B}) \subseteq I_{\A}$, $f_0(F_{\B}) \subseteq F_{\A}$, and  $\overline{\lambda}_{\A}\circ f^{\mathbb I} = \overline{\lambda}_{\B}$ hold for arbitrary cubical dimaps and not only for elementary ones.

\section{Naturality} \label{secNat}

It is not difficult to see that labeled homology is natural with respect to morphisms of HDAs. Here we establish that it also is natural with respect to cubical dimaps.

\subsection{The chain map induced by a cellular map}

Given precubical sets $P$ and $Q$ and a cellular map $f\colon |Q| \to |P|$, let $f_*\colon C_*(Q) \to C_*(P)$ be the unique chain map making the following diagram commute:
\[
\xymatrix{
	C_*(Q) \ar^{f_*}[r] \ar^{\cong}_{\cell }[d] & C_*(P) \ar_{\cong}^{\cell}[d] \\
	C^{CW}_*(|Q|) \ar_{f_*}[r]  & C^{CW}_*(|P|)
}
\]
Note that if $f$ is the geometric realization of a morphism of precubical sets $g\colon Q \to P$, then $f_* = g_*\colon C_*(Q) \to C_*(P)$. In Section \ref{fstar} we give an explicit description of $f_*$ for elementary cubical dimaps.

\subsection{Naturality of \texorpdfstring{$\zeta$}{}} Recall from Section \ref{cross} that given precubical sets $P$ and $Q$, $\zeta$ 
is the isomorphism of chain complexes \[ C_*(P)\otimes C_*(Q) \to C_*(P\otimes Q)\] given by \[x\otimes y \mapsto (x,y), \quad x\in P,\, y\in Q.\] 

\begin{lem} \label{zetanat}
The isomorphism $\zeta$ is natural with respect to cellular maps, i.e., for precubical sets $P$, $P'$, $Q$, and $Q'$, cellular maps $f\colon |P| \to |P'|$ and $g\colon |Q| \to |Q'|$, and the cellular map $f\times g \colon |P\otimes Q| = |P|\times |Q| \to |P'\otimes Q'| = |P'|\times |Q'|$, the following diagram of chain complexes is commutative:
\[
\xymatrix{
	C_*(P)\otimes C_*(Q) \ar^(0.55){\zeta}[r] \ar_{f_*\otimes g_*}[d] & C_*(P\otimes Q) \ar^{(f\times g)_*}[d] \\
	C_*(P')\otimes C_*(Q') \ar_(0.55){\zeta}[r]  & C_*(P'\otimes Q')
}
\] 
\end{lem}
\begin{sloppypar}
\begin{proof}
Given precubical sets $P$ and $Q$, let $\mu$ denote the unique chain map ${C^{CW}_*(|P|) \otimes C^{CW}_*(|Q|)} \to C^{CW}_*(|P| \times |Q|)$ making the following diagram commute:
\[
\xymatrix{
	C_*(P) \otimes C_*(Q) \ar^(0.55){\zeta}[r] \ar^{\cong}_{\cell \otimes \cell }[d] & C_*(P\otimes Q) \ar_{\cong}^{\psi_{P,Q *}\circ \cell} [d] \\
	C^{CW}_*(|P|) \otimes C^{CW}_*(|Q|)  \ar^(0.55){\mu}[r]  & C^{CW}_*(|P| \times |Q|)
}
\]
A straightforward verification shows that $\mu$ coincides on $C^{CW}_p(|P|) \otimes C^{CW}_q(|Q|)$ with the composite
\begin{align*}
	\MoveEqLeft{H_p(|P_{\leq p}|,|P_{<p}|)\otimes  H_q(|Q_{\leq q}|,|Q_{<q}|)}\\ &\xrightarrow{\times} H_{p+q}(|P_{\leq p}|\times |Q_{\leq q}|,|P_{\leq p}|\times |Q_{< q}|\cup |P_{< p}|\times |Q_{\leq q}|)\\ &\xrightarrow{j_*}  H_{p+q}((|P|\times |Q|)_{p+q},(|P|\times |Q|)_{p+q-1})\\
	&= C^{CW}_{p+q}(|P| \times |Q|)	
\end{align*}
where $j$ is the inclusion and the relative cross product $\times$ is defined using the map
\[{S_*(|P_{\leq p}|,|P_{<p}|)\otimes  S_*(|Q_{\leq q}|,|Q_{<q}|)}\to S_*(|P_{\leq p}|\times |Q_{\leq q}|,|P_{\leq p}|\times |Q_{< q}|\cup |P_{< p}|\times |Q_{\leq q}|)\]
induced by the Eilenberg-Zilber map (see Section \ref{cross}). It follows that $\mu$ is natural with respect to cellular maps, and this implies that $\zeta$ is natural with respect to cellular maps.
\end{proof}

\end{sloppypar}

It is clear that $\zeta$ is associative and can be iterated. Given precubical sets $P_1, \dots, P_n$, the chain isomorphism $C_*(P_1) \otimes \dots \otimes C_*(P_n) \to C_*(P_1 \otimes \dots \otimes P_n)$ will also be denoted by $\zeta$. It follows from Lemma \ref{zetanat} that this iterated $\zeta$ is natural with respect to cellular maps. 

\subsection{Interchange map of a square}
Consider the interchange map  
\[t\colon |\llbracket 0,1\rrbracket^{\otimes 2}| = |\llbracket 0,1\rrbracket|^{2} = [0,1]^2 \to |\llbracket 0,1\rrbracket^{\otimes 2}| = |\llbracket 0,1\rrbracket|^{2} = [0,1]^2,\quad  (x,y) \mapsto (y,x)\] 
and the induced chain map
\[t_*\colon C_*(\llbracket 0,1\rrbracket^{\otimes 2}) \to C_*(\llbracket 0,1\rrbracket^{\otimes 2}).\]

\begin{lem} \label{zetacom}
$t_*(\iota_2) = -\iota_2$.
\end{lem}

\begin{proof}
Since the boundary operator 
\[d\colon C_2(\llbracket 0,1\rrbracket^{\otimes 2}) \to C_1(\llbracket 0,1\rrbracket^{\otimes 2})\]
is injective, it is enough to show that $dt_*(\iota_2) = -d\iota_2$. In view of the commutative diagram 
\[
\xymatrix{
	C_*(\llbracket 0,1\rrbracket^{\otimes 2}) \ar^{t_*}[r] \ar^{\cong}_{\cell }[d] & C_*(\llbracket 0,1\rrbracket^{\otimes 2}) \ar_{\cong}^{\cell}[d] \\
	C^{CW}_*(|\llbracket 0,1\rrbracket^{\otimes 2}|) \ar_{t_*}[r]  & C^{CW}_*(|\llbracket 0,1\rrbracket^{\otimes 2}|),
}
\]	
it suffices to show that $dt_*(\cell(\iota_2)) = -d\cell(\iota_2)$. We have 
\[
d\cell(\iota_2) = \sum \limits_{i=1}^2 (-1)^i( \cell (d^0_i\iota_2) -  \cell (d^1_i\iota_2)) 
\]
and hence
\[
dt_*(\cell(\iota_2)) = \sum \limits_{i=1}^2 (-1)^i( t_*(\cell (d^0_i\iota_2)) -  t_*(\cell (d^1_i\iota_2))). 
\]
Identifying $ |(d^k_i\iota_2)_{\sharp}| = \delta^k_i\colon |\llbracket0,1\rrbracket| = [0,1] \to |\llbracket0,1\rrbracket^{\otimes 2}| = [0,1]^2$, we compute 
\[
t_*(\cell(d^k_i\iota_2))
= t_*([\delta^k_i])
= [t\circ \delta^k_i]
= [\delta^k_{3-i}]
= \cell(d^k_{3-i}\iota_2).
\]
It follows that 
\begin{align*}
dt_*(\cell(\iota_2)) 
&= \sum \limits_{i=1}^2 (-1)^i( t_*(\cell (d^0_i\iota_2)) -  t_*(\cell (d^1_i\iota_2)))\\
&= \sum \limits_{i=1}^2 (-1)^i( \cell (d^0_{3-i}\iota_2) -  \cell (d^1_{3-i}\iota_2))\\
&= -\sum \limits_{i=1}^2 (-1)^i( \cell (d^0_i\iota_2) -  \cell (d^1_i\iota_2))\\
&= -d\cell(\iota_2). \qedhere 
\end{align*} 
\end{proof}

\subsection{The chain map induced by an elementary cubical dimap} \label{fstar} Let $f$ be an elementary cubical dimap from a precubical set $Q$ to a precubical set $P$.

\begin{prop} \label{chainmap}
The chain map $f_*\colon C_*(Q) \to C_*(P)$ is given by 
\[
f_*(x) = \left\{ \begin{array}{ll}
f_0(x), & x \in Q_0,\\
\sgn(\sigma_x)\sum_{y \in (R_x)_n}x_{\flat}(y), &  x\in Q_n, n > 0.
\end{array}\right.
\]
\end{prop}

\begin{proof}
Let $x\in Q_0$. Since ${f\circ |x_{\sharp}|= |f_0(x)_{\sharp}|}$, the map  $f_*\colon C^{CW}_0(|Q|) \to C^{CW}_0(|P|)$  sends $\cell(x)$ to  $\cell(f_0(x))$. Hence $f_*\colon C_0(Q) \to C_0(P)$ sends $x$ to $f_0(x)$.

Let $n>0$ and $x \in Q_n$. Consider the following commutative diagram:
\[
\xymatrix{
C_*(\llbracket 0,1\rrbracket^{\otimes n}) \ar^(.55){t_{\sigma_x*}}[rr] \ar_{x_{\sharp*}}[d] & & C_*(\llbracket 0,1\rrbracket^{\otimes n}) \ar^(.55){(\phi_{x1} \times \cdots \times \phi_{xn})_*}[rr] & & C_*(R_x) \ar^{x_{\flat*}}[d] \\
C_*(Q) \ar_{f_*}[rrrr]& & & & C_*(P)
}
\]
We have $f_*(x) = f_*\circ x_{\sharp*}(\iota_n) = x_{\flat*} \circ (\phi_{x1} \times \cdots \times \phi_{xn})_* (t_{\sigma_x*}(\iota_n))$. We show first that $t_{\sigma_x*}(\iota_n) = \sgn(\sigma_x)\iota_n$. If $n=1$, there is nothing to show. If $n>1$, write $\sigma_x = {\sigma_1 \circ \dots \circ \sigma_m}$ where each $\sigma_j$ is a transposition of the form $(i\;\; i+1)$. Then $t_{\sigma_x} = t_{\sigma_m} \circ \dots \circ t_{\sigma_1}$ and $\sgn (\sigma_x) = (-1)^m$. It is thus enough to show that $t_{(i\;\; i+1) *}(\iota_n) = -\iota_n$. Since 
\[t_{(i\;\; i+1)} = \underbracket[0.5pt]{id \times \cdots \times id}_{\text{$i-1$ times}}\times t \times \underbracket[0.5pt]{id \times \cdots \times id}_{\text{$n-i-1$ times}}, \] Lemma \ref{zetanat} implies that the following diagram is commutative:
\[
\xymatrix{
	C_*(\llbracket 0,1 \rrbracket^{\otimes i-1})\otimes 	C_*(\llbracket 0,1 \rrbracket^{\otimes 2}) \otimes 	C_*(\llbracket 0,1 \rrbracket^{\otimes n-i-1}) \ar^(0.73){\zeta}[r] \ar_{id\otimes t_* \otimes id}[d] & C_*(\llbracket 0,1 \rrbracket^{\otimes n}) \ar^{t_{(i\;\;i+1) *}}[d] \\
	C_*(\llbracket 0,1 \rrbracket^{\otimes i-1})\otimes 	C_*(\llbracket 0,1 \rrbracket^{\otimes 2}) \otimes 	C_*(\llbracket 0,1 \rrbracket^{\otimes n-i-1}) \ar_(0.73){\zeta}[r]  & C_*(\llbracket 0,1 \rrbracket^{\otimes n})
}
\] 
Thus, by Lemma \ref{zetacom},  
\begin{align*}
t_{(i\;\; i+1) *}(\iota_n) &= t_{(i\;\; i+1) *} \circ \zeta (\iota_{i-1} \otimes \iota_2 \otimes \iota_{n-i-1})\\
&= \zeta \circ (id\otimes t_* \otimes id) (\iota_{i-1} \otimes \iota_2 \otimes \iota_{n-i-1})\\
&= \zeta (\iota_{i-1} \otimes (-\iota_2) \otimes \iota_{n-i-1})\\
&= - \iota_n.
\end{align*}

We next compute $(\phi_{x1} \times \cdots \times \phi_{xn})_*(\iota_n)$. Suppose that $R_x = \llbracket 0,{l_1} \rrbracket\otimes \cdots \otimes \llbracket 0,{l_n} \rrbracket$. By Lemma \ref{zetanat}, we have the following commutative diagram:
\[
\xymatrix{
C_*(\llbracket 0,1\rrbracket)^{\otimes n} \ar^(.4){\phi_{x1*} \otimes \cdots \otimes \phi_{xn*}}[rrr] \ar_{\zeta}[d]& & & C_*(\llbracket 0,{l_1} \rrbracket) \otimes \cdots \otimes C_*(\llbracket 0,{l_n} \rrbracket) \ar^{\zeta}[d] \\
C_*(\llbracket 0,1\rrbracket^{\otimes n}) \ar_(0.4){(\phi_{x1} \times \cdots \times \phi_{xn})_*}[rrr]  & & & C_*(\llbracket 0,{l_1} \rrbracket \otimes \cdots \otimes \llbracket 0,{l_n} \rrbracket)
}
\]
One easily computes that $\phi_{xi*} \colon C_*(\llbracket 0,1 \rrbracket) \to C_*(\llbracket0,l_i\rrbracket)$ sends $\iota_1$ to $\sum_{j_i=0}^{l_i-1}[j_i,j_i+1]$. Hence 
\begin{align*}
(\phi_{x1} \times \cdots \times \phi_{xn})_*(\iota_n) &= (\phi_{x1} \times \cdots \times \phi_{xn})_*\circ \zeta (\iota_1 \otimes \dots \otimes \iota_1)\\
&= \zeta \circ (\phi_{x1*} \otimes \cdots \otimes \phi_{xn*})(\iota_1 \otimes \dots \otimes \iota_1)\\
&= \zeta (\phi_{x1*}(\iota_1) \otimes \cdots \otimes \phi_{xn*}(\iota_1))\\
&= \zeta \left(\left(\sum_{j_1=0}^{l_1-1}[j_1,j_1+1]\right) \otimes \cdots \otimes \left(\sum_{j_n=0}^{l_n-1}[j_n,j_n+1]\right)\right )\\ 
&= \sum_{j_1=0}^{l_1-1}\dots \sum_{j_n=0}^{l_n-1} \zeta ([j_1,j_1+1]\otimes  \dots \otimes [j_n,j_n+1])\\
&= \sum_{j_1=0}^{l_1-1}\dots \sum_{j_n=0}^{l_n-1} ([j_1,j_1+1], \dots,[j_n,j_n+1])\\ 
&= \sum_{y \in (R_x)_n}y.
\end{align*}
Finally,   
$f_*(x) = x_{\flat*} \circ (\phi_{x1} \times \cdots \times \phi_{xn})_* (t_{\sigma_x*}(\iota_n)) = \sgn(\sigma_x)\sum_{y \in (R_x)_n}x_{\flat}(y)$.
\end{proof}

\subsection{Naturality of labeled homology}

We are now ready to establish the naturality of labeled homology with respect to cubical dimaps.

\begin{theor} \label{natweak}
Let $f$ be a cubical dimap from a $\Sigma^*\mbox{-}$HDA $\B$ to a $\Sigma^*\mbox{-}$HDA $\A$. Then $\mathfrak{l}_{B}^n = \mathfrak{l}_{\A}^n\circ f_*$ for all $n \geq 0$. Consequently, $\ell_{\B} = \ell_{\A}\circ f_*$ in all degrees.
\end{theor}

\begin{proof}
We may suppose that $f$ is elementary. For any vertex $v\in Q_0$, $\mathfrak{l}_{\B}^0(v) = 1 = \mathfrak{l}_{\A}^0 (f_0(v)) = \mathfrak{l}_{\A}^0 \circ f_*(v)$. Suppose that $n > 0$. Consider an element $x\in Q_n$, and let $R_x = \llbracket 0,{l_1} \rrbracket \otimes \cdots \otimes \llbracket 0,{l_n} \rrbracket$. For $i \in \{1, \dots, n\}$ and $j_i \in \{0, \dots, l_i-1\}$, consider the edge
\[
b_i^{j_i} = (l_1, \dots ,l_{i-1}, [j_i,j_i+1],l_{i+1}, \dots,l_n)
\] 
of $R_x$. We show first that the following diagram is commutative:
\[
\xymatrix{
|\llbracket 0,1\rrbracket| \ar^(0.5){\phi_{xi}}[rrrr] \ar_{|(e^0_{\sigma_x(i)}\iota_n)_{\sharp}|}[d] & & & & |\llbracket 0,{l_i} \rrbracket| \ar^{|b^{0}_{i\sharp} \cdots b^{l_i-1}_{i\sharp}|}[d] \\
|\llbracket 0,1\rrbracket^{\otimes n}| \ar_{t_{\sigma_x}}[rr]  & & |\llbracket 0,1\rrbracket^{\otimes n}| \ar_(0.4){\phi_{x1} \times \cdots \times \phi_{xn}}[rr] & & |\llbracket 0,{l_1} \rrbracket \otimes \cdots \otimes \llbracket 0,{l_n} \rrbracket| 
}
\]
We have $(e^0_{\sigma_x(i)}\iota_n)_{\sharp}(\iota_1) = e^0_{\sigma_x(i)}\iota_n = (1,\dots , 1, \underset{\sigma_x(i)}{[0,1]},1, \dots, 1)$ and hence
\begin{align*}
\MoveEqLeft{(\phi_{x1} \times \cdots \times \phi_{xn}) }\circ t_{\sigma_x} \circ |(e^0_{\sigma_x(i)}\iota_n)_{\sharp}| ([\iota_1,t])\\ 
&=  {(\phi_{x1} \times \cdots \times \phi_{xn}) }\circ t_{\sigma_x} ([(1,\dots , 1, \underset{\sigma_x(i)}{[0,1]},1, \dots, 1),t])\\
&=   {(\phi_{x1} \times \cdots \times \phi_{xn}) }\circ t_{\sigma_x} (1,\dots , 1, \underset{\sigma_x(i)}{t},1, \dots, 1)\\
&=   {\phi_{x1} \times \cdots \times \phi_{xn} } (1,\dots , 1, \underset{i}{t},1, \dots, 1)\\
&=  (l_1,\dots , l_{i-1}, \phi_{xi}(t),l_{i+1}, \dots, l_{n}).
\end{align*}
Let $j_i\in \{0,\dots,l_{i-1}\}$ and $s\in [0,1]$ be numbers such that $\phi_{xi}(t) = j_i + s = [[j_i,j_i+1],s]$. Then 
\begin{align*}
\MoveEqLeft{(\phi_{x1} \times \cdots \times \phi_{xn}) }\circ t_{\sigma_x} \circ |(e^0_{\sigma_x(i)}\iota_n)_{\sharp}| ([\iota_1,t])\\  
&= [(l_1,\dots , l_{i-1}, [j_i,j_i+1],l_{i+1}, \dots, l_{n}),s] \\
&= [b_i^{j_i},s]\\
&= [b^{0}_{i\sharp} \cdots b^{l_i-1}_{i\sharp}([j_i,j_i+1]),s]\\
&= |b^{0}_{i\sharp} \cdots b^{l_i-1}_{i\sharp}|([[j_i,j_i+1],s])\\
&= |b^{0}_{i\sharp} \cdots b^{l_i-1}_{i\sharp}|\circ \phi_{xi}([\iota_1,t]).
\end{align*}
Since $x_{\sharp} \circ (e^0_{\sigma_x(i)}\iota_n)_{\sharp}(\iota_1)  = x_{\sharp} (e^0_{\sigma_x(i)}\iota_n) =  e^0_{\sigma_x(i)}x_{\sharp} (\iota_n) = e^0_{\sigma_x(i)}x$, we have  \[{x_{\sharp} \circ (e^0_{\sigma_x(i)}\iota_n)_{\sharp}} = (e^0_{\sigma_x(i)}x)_{\sharp}\]  
and therefore
\[f^{\mathbb I}((e^0_{\sigma_x(i)}x)_{\sharp}) = x_{\flat} \circ b^{0}_{i\sharp} \cdots b^{l_i-1}_{i\sharp} = x_{\flat}(b_{i}^0)_{\sharp} \cdots x_{\flat}(b_{i}^{l_i-1})_{\sharp}.\]
Hence
\[
\lambda_{\B} (e^0_{\sigma_x(i)}x) =  \overline{\lambda}_{\B}((e^0_{\sigma_x(i)}x)_{\sharp}) = \overline{\lambda}_{\A}(f^{\mathbb I}((e^0_{\sigma_x(i)}x)_{\sharp}))  =  \lambda_{\A}(x_{\flat}(b_{i}^0)) \cdots \lambda_{\A}(x_{\flat}(b_{i}^{l_i-1})).
\]
Consequently,
\[
\mathfrak{l}_{\B}^1(e^0_{\sigma_x(i)}x)  = \sum_{j_i=0}^{l_i-1}  \sum_{r_i^{j_i}=1}^{|\lambda_{\A}(x_{\flat}(b_{i}^{j_i}))|} \lambda_{\A}(x_{\flat}(b_{i}^{j_i}))_{r_i^{j_i}} = \sum_{j_i=0}^{l_i-1}  \mathfrak{l}_{\A}^1(x_{\flat}(b_{i}^{j_i})).
\]
Thus
\begin{align*}
\mathfrak{l}_{\B}^n(x) &= \mathfrak{l}_{\B}^1(e^0_1x)\wedge \cdots \wedge \mathfrak{l}_{\B}^1(e^0_nx)\\
&= \sgn(\sigma_x) \mathfrak{l}_{\B}^1(e^0_{\sigma_x(1)}x)\wedge \cdots \wedge \mathfrak{l}_{\B}^1(e^0_{\sigma_x(n)}x)\\
&= \sgn(\sigma_x) \left(\sum_{j_1=0}^{l_1-1}  \mathfrak{l}_{\A}^1(x_{\flat}(b_{1}^{j_1})) \right) \wedge  \cdots \wedge \left(\sum_{j_n=0}^{l_n-1}  \mathfrak{l}_{\A}^1(x_{\flat}(b_{n}^{j_n})) \right)\\
&= \sgn(\sigma_x) \sum_{j_1=0}^{l_1-1} \cdots \sum_{j_n=0}^{l_n-1} \mathfrak{l}_{\A}^1(x_{\flat}(b_{1}^{j_1})) \wedge  \cdots \wedge \mathfrak{l}_{\A}^1(x_{\flat}(b_{n}^{j_n})).
\end{align*}

Consider the element $y = ([j_1,j_1+1], \dots ,[j_n,j_n+1]) \in (R_x)_n$.	We have 
\begin{align*}
\mathfrak{l}_{\A}^n(x_{\flat }(y)) 
&= \mathfrak{l}_{\A}^1(e^0_1x_{\flat }(y)) \wedge \cdots \wedge \mathfrak{l}_{\A}^1(e^0_nx_{\flat }(y))\\
&= \mathfrak{l}_{\A}^1(x_{\flat }(e^0_1y)) \wedge \cdots \wedge\mathfrak{l}_{\A}^1(x_{\flat }(e^0_ny)).
\end{align*}
Now, by Lemma \ref{paralleledges},
\begin{align*}
\lambda_{\A}(x_{\flat}(e^0_iy)) 
&= \lambda_{\A}(x_{\flat}(j_1+1, \dots ,j_{i-1}+1, [j_i,j_i+1], j_{i+1}+1, \dots ,j_n+1))\\
&= \lambda_{\A}(x_{\flat}(l_1, \dots ,l_{i-1}, [j_i,j_i+1], l_{i+1}, \dots ,l_n))\\
&= \lambda_{\A}(x_{\flat}(b_i^{j_i})).
\end{align*} 
Thus $\mathfrak{l}_{\A}^1(x_{\flat }(e^0_iy)) = \mathfrak{l}_{\A}^1(x_{\flat}(b_i^{j_i}))$ 
and therefore 
\[
\mathfrak{l}_{\A}^n(x_{\flat }(y)) 
= \mathfrak{l}_{\A}^1(x_{\flat}(b_1^{j_1})) \wedge \cdots \wedge \mathfrak{l}_{\A}^1(x_{\flat}(b_n^{j_n})).
\]
Using Proposition \ref{chainmap}, we obtain 
\[
\mathfrak{l}_{\B}^n (x) = \sgn(\sigma_x) \sum_{y \in (R_x)_n} \mathfrak{l}_{\A}^n(x_{\flat }(y)) = \mathfrak{l}_{\A}^n(f_*(x)). \qedhere 
\]
\end{proof}

As an immediate consequence of Theorem \ref{natweak}, we note the following result: 

\begin{cor}
	Let $\A$ be a $\Sigma^*\mbox{-}$HDA such that the label of some homology class of degree $> 0$ is nonzero. Then there does not exist any cubical dimap from $\A$ to a contractible $\Sigma^*\mbox{-}$HDA.
\end{cor}

\section{Applications} \label{secApp}

In this section we discuss two potential applications of labeled homology in concurrency theory.

\subsection{Independence} \label{secIndep}

An important global question about a given concurrent system is which subsystems are independent from each other. Let us define here that the $M\mbox{-}$HDAs $\B_1, \dots, \B_n$ are  \emph{independent} in the $M\mbox{-}$HDA $\A$ if there exist a subautomaton $\B$ of $\A$ and a cubical dimap of $M\mbox{-}$HDAs ${|\B_1\otimes \dots \otimes \B_n|} \to |\B|$ that is a homeomorphism. This concept of independence is a variant of the one considered in \cite[5.21]{FGHMR}.

\begin{prop} \label{indep}
Suppose that the $\Sigma^*\mbox{-}$HDAs $\B_i$ $(i = 1, \dots, n)$ are independent in the $\Sigma^*\mbox{-}$HDA $\A$, and consider homology classes $\beta_i\in H_*(\B_i)$ $(i = 1, \dots, n)$. Then there exists a homology class $\alpha \in H_*(\A)$ such that \[\ell_{\A}(\alpha) = {\ell_{\B_1}(\beta_1) \wedge \dots \wedge \ell_{\B_n}(\beta_n)}.\]
\end{prop}

\begin{proof}
This follows immediately from Theorems \ref{crosslabel} and \ref{natweak}.
\end{proof}

\begin{ex}
Let us consider the following version of the classical \emph{dining philosophers} system \cite{DijkstraHierarchical}: Four philosophers sit at a round table having dinner. A chopstick is placed to the left of each philosopher. In order to eat, a philosopher needs this chopstick and also the one on the right. Each philosopher alternately thinks and eats according to the following instructions: 
\begin{itemize}
	\item Think.
	\item Wait until the left chopstick is free and pick it up.
	\item Wait until the right chopstick is free and pick it up.
	\item Eat.
	\item Put the left chopstick down.
	\item Put the right chopstick down. 
	\item Repeat the procedure from the beginning.
\end{itemize}
At the beginning of the dinner, all chopsticks are free. 

The reachable part of the state space of the dining philosophers system may be modeled by a $\Sigma ^*\mbox{-}$HDA $\A$ where the alphabet $\Sigma$ is given by 
\[\Sigma = \{\mathtt{think_i}, \mathtt{eat_i}, \mathtt{pick\_l_i}, \mathtt{put\_l_i}, \mathtt{pick\_r_i},\mathtt{put\_r_i} \,|\, \mathtt{i} \in \{0,1,2,3\}\}.\]
Here the indexes correspond, of course, to the four philosophers. The HDA $\A$ is of dimension 4 and has 465 vertices, 1508 edges, 1766 2-cubes, 884 3-cubes, and 160 4-cubes. 

For each $i \in\{0,1,2,3\}$, the HDA $\A$ contains a subautomaton $\B_i$ corresponding to the subsystem where only the $i$th philosopher has dinner and the others remain inactive in their initial states. The HDA $\B_i$ is a directed circle and has a 1$\mbox{-}$dimensional homology generator with label
\[
l_i = \mathtt{think_i} + \mathtt{pick\_l_i} + \mathtt{pick\_r_i} + \mathtt{eat_i} + \mathtt{put\_l_i} + \mathtt{put\_r_i}.
\]
By Theorem \ref{natweak}, it follows that $\A$ has four 1-dimensional homology classes with labels $l_0$, $l_1$, $l_2$, and $l_3$, respectively. Since these labels are linearly independent, so are the homology classes.

If two philosophers, say philosophers $i$ and $j$, are not neighbors at the table, they eat independently because they do not share any chopstick. Consequently, $\B_i$ and $\B_j$ are independent in $\A$. By Proposition \ref{indep}, it follows that $\A$ has a 2-dimensional homology class with label $l_i\wedge l_j$. Since the other two philosophers are also not neighbors, $\A$ has a second 2-dimensional homology class of this type, and the two homology classes are linearly independent because their labels are.

Computing the $\Z_2$-homology of $\A$ using the software CHomP \cite{chomp}, one obtains that $H_{n}(\A) = 0$ for $n>2$, $\dim H_0(\A) = 1$, $\dim H_1(\A) = 4$, and $\dim H_2(\A) = 2$. We have thus completely determined the labeled  $\Z_2\mbox{-}$homology of $\A$.

Suppose that philosophers $i$ and $j$ are neighbors at the table. Then, as one might expect, the HDAs $\B_i$ and $\B_j$ are not independent in $\A$. Otherwise, by Proposition \ref{indep}, $l_i\wedge l_j$ would be the label of a 2-dimensional homology class of $\A$, which is not the case. Note that the reason for the dependence of $\B_i$ and $\B_j$ is not just that philosophers $i$ and $j$ share a resource. Indeed, as shows Example \ref{specex} below, processes that share resources may very well be independent.
\end{ex}

\subsection{Specification} \label{secSpec}
\begin{sloppypar}
Consider the practical problem to decide whether an $M\mbox{-}$HDA $\A$ implements a certain specification that is given by another $M\mbox{-}$HDA $\mathcal S$. Let us say here that $\A$ \emph{implements} $\mathcal S$ if there exists a finite sequence of cubical dimaps of $M\mbox{-}$HDAs
\[
|\A| \to \cdot \xleftarrow{\sim} \cdot \to \cdots  \xleftarrow{\sim} \cdot \to {|\mathcal S|}
\]
where the cubical dimaps indicated by $\xleftarrow{\sim}$ are homotopy equivalences that could also be required to preserve relevant constructions such as the trace category and the homology graph (see \cite{hgraph,weakmor,topabs}). This implementation relation can be seen as a kind of simulation preorder that permits one to relate HDAs of different refinement levels. 
\end{sloppypar}

\begin{prop} \label{spec}
Let $\A$ and $\mathcal S$ be $\Sigma^ *$-HDAs such that $\A$ implements $\mathcal S$. Then for each homology class $\alpha \in H_*(\A)$, there exists a homology class $\beta \in H_*(\mathcal S)$ such that $\ell_{\mathcal S}(\beta) = \ell_{\A}(\alpha)$.
\end{prop}

\begin{proof}
This follows immediately from Theorem \ref{natweak}.
\end{proof}

\begin{ex} \label{specex}
As an example, consider the task to implement a locking mechanism for two processes by incrementing and decrementing an integer variable $\mathtt{x}$ using the operators $\mathtt{x{\scriptstyle{++}}}$ and $\mathtt{x{\scriptstyle{--}}}$, respectively. These operators are to be interpreted in such a way that $\mathtt{x{\scriptstyle{++}}}$ means to acquire the lock and $\mathtt{x{\scriptstyle{--}}}$ means to release it. In order to specify the task more formally, consider the alphabet 
\[
\Sigma = \{\mathtt{x{\scriptstyle{++}}_0}, \mathtt{x{\scriptstyle{++}}_1}, \mathtt{x{\scriptstyle{--}}_0}, \mathtt{x{\scriptstyle{--}}_1}\},
\]
where the indexes in the instructions are the process IDs, and the $\Sigma ^*\mbox{-}$HDA $\mathcal S$ given by a wedge of two directed circles, one labeled with the string $\mathtt{x{\scriptstyle{++}}_0};\mathtt{x{\scriptstyle{--}}_0}$ and the other with the string $\mathtt{x{\scriptstyle{++}}_1};\mathtt{x{\scriptstyle{--}}_1}$. Let us suppose now we had the (bad) idea to implement this locking mechanism through a concurrent program where both processes execute a loop consisting of the instruction sequence $\mathtt{x{\scriptstyle{++}}};\mathtt{x{\scriptstyle{--}}}$. Assuming that $\mathtt{x}$ is initially $0$, the HDA $\A$ corresponding to the reachable part of the state space of this program is a directed torus, and there exists a 2-dimensional homology class with label
\[
(\mathtt{x{\scriptstyle{++}}_0} + \mathtt{x{\scriptstyle{--}}_0})\wedge (\mathtt{x{\scriptstyle{++}}_1}+ \mathtt{x{\scriptstyle{--}}_1}),
\]
which is nonzero. Since the 2-dimensional homology of $\mathcal S$ is trivial, Proposition \ref{spec} implies that $\A$ does not implement $\mathcal S$. The error in our program is, of course, that we forgot to include a guard condition that forces the processes to wait until the variable $\mathtt{x}$ is $0$ before they increment it. 
\end{ex}

\bibliography{refs}
\bibliographystyle{amsalpha}
\end{document}